\documentclass[a4paper]{article}
\usepackage{enumitem}
\usepackage{calc}

\usepackage{lmodern} 

\usepackage{graphicx} 

\usepackage{amsmath, accents}
\usepackage{amssymb,tikz}
\usepackage{pict2e}
\usepackage{amsthm}
\usepackage{amscd}
\usepackage{mathrsfs}
\usepackage{todonotes}
\usetikzlibrary{calc} 

\usepackage{yfonts}

\usepackage{marvosym}
\usepackage[percent]{overpic}

\newtheorem{theorem}{Theorem} [section]

\newtheorem{lemma}[theorem]{Lemma}

\newtheorem{remark}[theorem]{Remark}
\theoremstyle{definition}

\newcommand{\R}{\mathbb{R}}
\newcommand{\N}{\mathbb{N}}
\newcommand{\re}{\text{\upshape Re\,}}
\newcommand{\im}{\text{\upshape Im\,}}

\let\oldbibliography\thebibliography
\renewcommand{\thebibliography}[1]{\oldbibliography{#1}
\setlength{\itemsep}{-0.5pt}}

\def\XXint#1#2#3{{\setbox0=\hbox{$#1{#2#3}{\int}$}
\vcenter{\hbox{$#2#3$}}\kern-.5\wd0}}

\usepackage{tikz}
\usetikzlibrary{arrows}
\usetikzlibrary{decorations.pathmorphing}
\usetikzlibrary{decorations.markings}
\usetikzlibrary{patterns}
\usetikzlibrary{automata}
\usetikzlibrary{positioning}
\usepackage{tikz-cd}
\tikzset{->-/.style={decoration={
				markings,
				mark=at position #1 with {\arrow{latex}}},postaction={decorate}}}
	
	\tikzset{-<-/.style={decoration={
				markings,
				mark=at position #1 with {\arrowreversed{latex}}},postaction={decorate}}}

\usetikzlibrary{shapes.misc}\tikzset{cross/.style={cross out, draw, 
         minimum size=2*(#1-\pgflinewidth), 
         inner sep=0pt, outer sep=0pt}}

\tikzset{
	master/.style={
		execute at end picture={
			\coordinate (lower right) at (current bounding box.south east);
			\coordinate (upper left) at (current bounding box.north west);
		}
	},
	slave/.style={
		execute at end picture={
			\pgfresetboundingbox
			\path (upper left) rectangle (lower right);
		}
	}
}

\usepackage{pgfplots}
	
\allowdisplaybreaks	

\numberwithin{equation}{section}

\def\bigO{{\cal O}}

\makeatletter
\newcommand{\oset}[3][0ex]{%
  \mathrel{\mathop{#3}\limits^{
    \vbox to#1{\kern-2\ex@
    \hbox{$\scriptstyle#2$}\vss}}}}
\makeatother

\usepackage[colorlinks=true]{hyperref}
\hypersetup{urlcolor=blue, citecolor=red, linkcolor=blue}

\begin{document}
\title{\vspace{-1cm} A point process on the unit circle with mirror-type interactions}
\author{Christophe Charlier\footnote{Centre for Mathematical Sciences, Lund University, 22100 Lund, Sweden. e-mail: christophe.charlier@math.lu.se
}}

\maketitle

\begin{abstract}
We consider the point process 
\begin{align*}
\frac{1}{Z_{n}}\prod_{1 \leq j < k \leq n} |e^{i\theta_{j}}-e^{-i\theta_{k}}|^{\beta}\prod_{j=1}^{n} d\theta_{j}, \qquad \theta_{1},\ldots,\theta_{n} \in (-\pi,\pi], \quad \beta > 0,
\end{align*}
where $Z_{n}$ is the normalization constant. The feature of this process is that the points $e^{i\theta_{1}},\ldots,e^{i\theta_{n}}$ interact with the mirror points reflected over the real line $e^{-i\theta_{1}},\ldots,e^{-i\theta_{n}}$. 

We study smooth linear statistics of the form $\sum_{j=1}^{n}g(\theta_{j})$ as $n \to \infty$, where $g$ is $2\pi$-periodic. We prove that a wide range of asymptotic scenarios can occur: depending on $g$, the leading order fluctuations around the mean can (i) be of order $n$ and purely Bernoulli, (ii) be of order $1$ and purely Gaussian, (iii) be of order $1$ and purely Bernoulli, or (iv) be of order $1$ and of the form $BN_{1}+(1-B)N_{2}$, where $N_{1},N_{2}$ are two independent Gaussians and $B$ is a Bernoulli that is independent of $N_{1}$ and $N_{2}$. The above list is not exhaustive: the fluctuations can be of order $n$, of order $1$ or $o(1)$, and other random variables can also emerge in the limit. 

We also obtain large $n$ asymptotics for $Z_{n}$ (and some generalizations), up to and including the term of order $1$.

Our proof is inspired by a method developed by McKay and Wormald \cite{McKayWormald} to estimate related $n$-fold integrals.
\end{abstract}



\section{Introduction}

Gibbs measures are models for the behavior of interacting particle systems and are widely used to describe phenomena from statistical mechanics \cite{DVJ2008}. While particle interactions can be either attractive or repulsive, the mathematical literature has so far focused more on repulsive point processes. A well-studied example of repulsive point process is the circular $\beta$-ensemble (C$\beta$E), defined by 
\begin{align}\label{CbetaE}
\frac{1}{Z_{n,1}}\prod_{1 \leq j < k \leq n} |e^{i\theta_{j}}-e^{i\theta_{k}}|^{\beta}\prod_{j=1}^{n} d\theta_{j}, \qquad \theta_{1},\ldots,\theta_{n} \in (-\pi,\pi],
\end{align}
where $Z_{n,1}$ is the normalization constant that turns \eqref{CbetaE} into a probability measure. The C$\beta$E models $n$ particles on the unit circle interacting through the two–dimensional Coulomb law at inverse temperature $\beta>0$, see also e.g. \cite[Chapter 2]{For} for further background.

\medskip By switching a sign in \eqref{CbetaE}, either by replacing $e^{i\theta_{k}}$ by $e^{-i\theta_{k}}$ or by turning the minus sign in \eqref{CbetaE} into a plus sign, we obtain the following two point processes
\begin{align}
& \frac{1}{Z_{n,2}}\prod_{1 \leq j < k \leq n} |e^{i\theta_{j}}-e^{-i\theta_{k}}|^{\beta}\prod_{j=1}^{n} d\theta_{j}, \qquad \theta_{1},\ldots,\theta_{n} \in (-\pi,\pi], \label{new point process intro} \\
& \frac{1}{Z_{n,3}}\prod_{1 \leq j < k \leq n} |e^{i\theta_{j}}+e^{i\theta_{k}}|^{\beta}\prod_{j=1}^{n} d\theta_{j}, \qquad \theta_{1},\ldots,\theta_{n} \in (-\pi,\pi], \label{AntipodalType}
\end{align}
where $Z_{n,2}$ and $Z_{n,3}$ are normalizing constants. 

\medskip (Switching simultaneously the two aforementioned signs in \eqref{CbetaE} yields
\begin{align*}
& \frac{1}{Z_{n,2}}\prod_{1 \leq j < k \leq n} |e^{i\theta_{j}}+e^{-i\theta_{k}}|^{\beta}\prod_{j=1}^{n} d\theta_{j}, \qquad \theta_{1},\ldots,\theta_{n} \in (-\pi,\pi],
\end{align*}
which is equivalent to \eqref{new point process intro} by applying a rotation by $90$ degrees. Hence, in what follows, we focus the discussion on \eqref{new point process intro}--\eqref{AntipodalType}.)

\medskip One of the interesting properties of \eqref{new point process intro} and \eqref{AntipodalType} is that they are attractive, in the sense that the densities of \eqref{new point process intro} and \eqref{AntipodalType} are maximized only when all points coincide on the unit circle. Indeed, it is easy to check that for $n\geq 3$, only two configurations maximize \eqref{new point process intro}, namely $(e^{i\theta_{1}},\ldots,e^{i\theta_{n}})=(i,\ldots,i)$ and $(e^{i\theta_{1}},\ldots,e^{i\theta_{n}})=(-i,\ldots,-i)$. Also, for any $n\geq 1$, the set of point configurations maximizing \eqref{AntipodalType} is given by $\{(e^{i\theta_{1}},\ldots,e^{i\theta_{n}})=(e^{i\theta}, \ldots, e^{i\theta}): \theta \in (-\pi,\pi]\}$. 

\medskip The simple structure of the point processes  \eqref{new point process intro} and \eqref{AntipodalType} in combination with their attractive interactions suggests that they could be relevant for several applications. However, we do not aim at finding such applications here; instead, we take the point of view that \eqref{new point process intro} and \eqref{AntipodalType} are sufficiently interesting as simple examples of attractive point processes to merit further study. Furthermore, as the number of points $n$ gets large, \eqref{new point process intro} can be rigorously analyzed using techniques from \cite{McKayWormald} (more on this below), and \eqref{AntipodalType} can be analyzed using techniques from \cite{McKay, IsaevMcKay}. This connection with the works \cite{McKayWormald, McKay, IsaevMcKay} makes these point processes also valuable from a mathematical perspective.

\medskip One can also view \eqref{new point process intro} and \eqref{AntipodalType} as repulsive point processes of a new type, where the repulsion occurs between the points $e^{i\theta_{1}},\ldots,e^{i\theta_{n}}$ and some ``image points", but not between the points themselves. Indeed, the feature of the point process \eqref{new point process intro} is that the points $e^{i\theta_{1}},\ldots,e^{i\theta_{n}}$ are repelled by the image points $e^{-i\theta_{1}},\ldots,$ $e^{-i\theta_{n}}$ obtained by reflection over the real line. For this reason, we call \eqref{new point process intro} a point process with ``mirror-type interactions" (where we think of the real line as the mirror). Similarly, we call \eqref{AntipodalType} a point process with ``antipodal interactions" since the points $e^{i\theta_{1}},\ldots,e^{i\theta_{n}}$ are repelled by the antipodal points $-e^{i\theta_{1}},\ldots,$ $-e^{i\theta_{n}}$. 

\medskip There are significant differences in the asymptotic analysis as $n\to \infty$ of \eqref{new point process intro} and \eqref{AntipodalType}. We discuss these differences at the end of this section. In this work we focus on the point process \eqref{new point process intro}; the point process \eqref{AntipodalType} is studied in the companion paper \cite{C ReflectionPoint}.


\medskip In what follows, we let $(\theta_{1},\ldots,\theta_{n})$ be distributed as in \eqref{new point process intro}, and for convenience we set $Z_{n}:=Z_{n,2}$. Our first result makes precise the idea that for large $n$ only the point configurations that are close to either $(i,\ldots,i)$ or $(-i,\ldots,-i)$ are likely to occur.
\begin{theorem}\label{thm:prob}
Fix $\beta > 0$. For any $\epsilon \in (0,\frac{1}{14})$, there exists $c>0$ such that, for all large enough $n$,
\begin{multline*}
\hspace{-0.2cm}\mathbb{P}\bigg( \hspace{-0.05cm} \Big( |e^{i\theta_{j}}-i|\leq n^{-\frac{1}{2}+\epsilon} \hspace{-0.065cm} \mbox{ for all } j \hspace{-0.05cm} \in \hspace{-0.05cm} \{1,\ldots,n\} \Big) \hspace{-0.1cm} \mbox{ or } \hspace{-0.05cm} \Big( |e^{i\theta_{j}}+i|\leq n^{-\frac{1}{2}+\epsilon} \hspace{-0.065cm} \mbox{ for all } j \hspace{-0.05cm} \in \hspace{-0.05cm} \{1,\ldots,n\} \Big) \hspace{-0.05cm} \bigg) \hspace{-0.05cm} \geq 1-e^{-cn^{2\epsilon}}.
\end{multline*} 
\end{theorem}

\begin{figure}[h]
\begin{center}
\begin{tikzpicture}[master]
\node at (0,0) {\includegraphics[width=3.5cm]{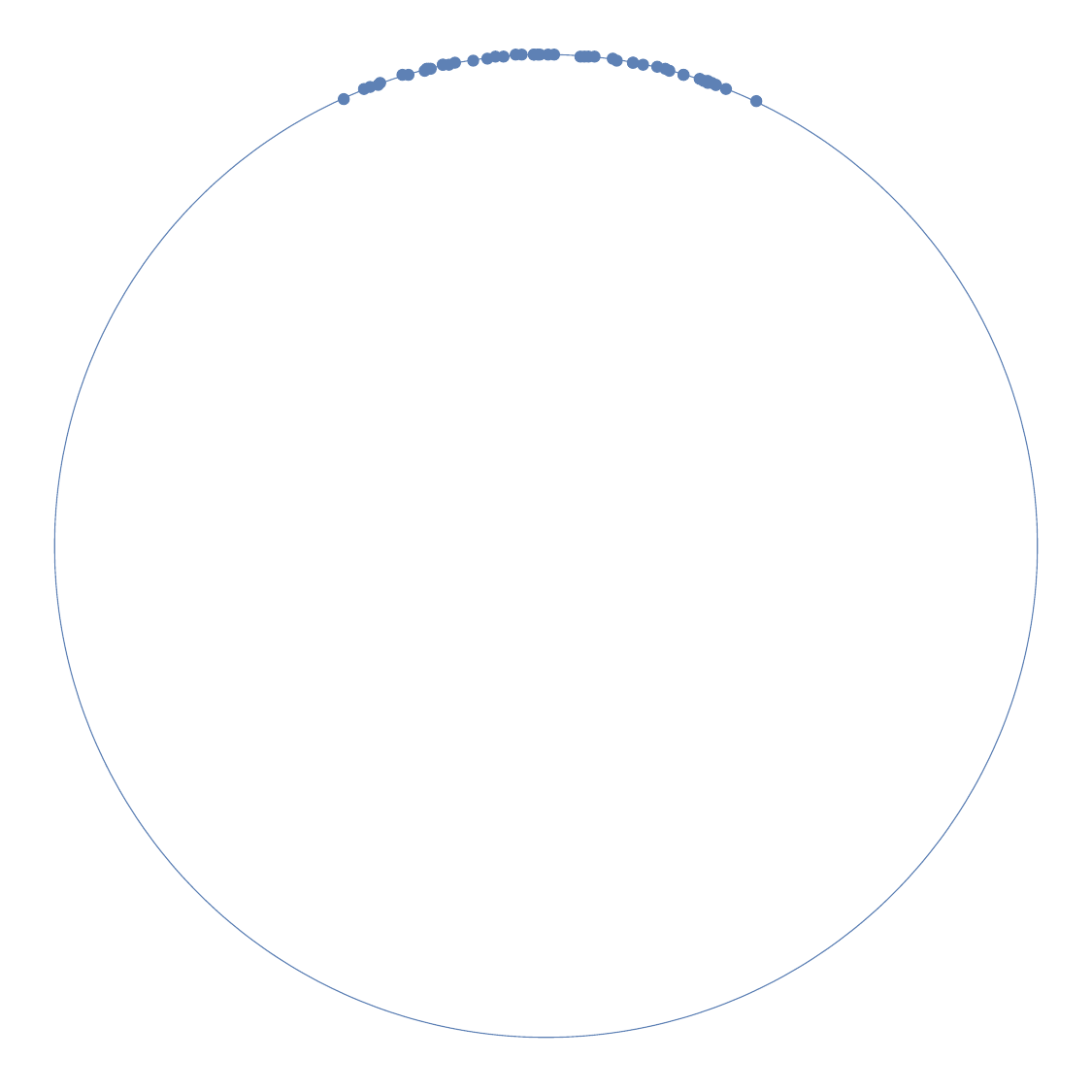}};
\draw[black,line width=0.15 mm,-<-=0.01,->-=1] ([shift=(60:1.35cm)]0,0) arc (60:120:1.35cm);
\node at (0,0.97) {$\bigO(n^{-\frac{1}{2}+\epsilon})$};
\draw[fill] (0,0) circle (0.9pt);
\node at (0.2,0) {$0$};
\end{tikzpicture} \hspace{1.5cm}
\begin{tikzpicture}[slave]
\node at (0,0) {\includegraphics[width=3.5cm]{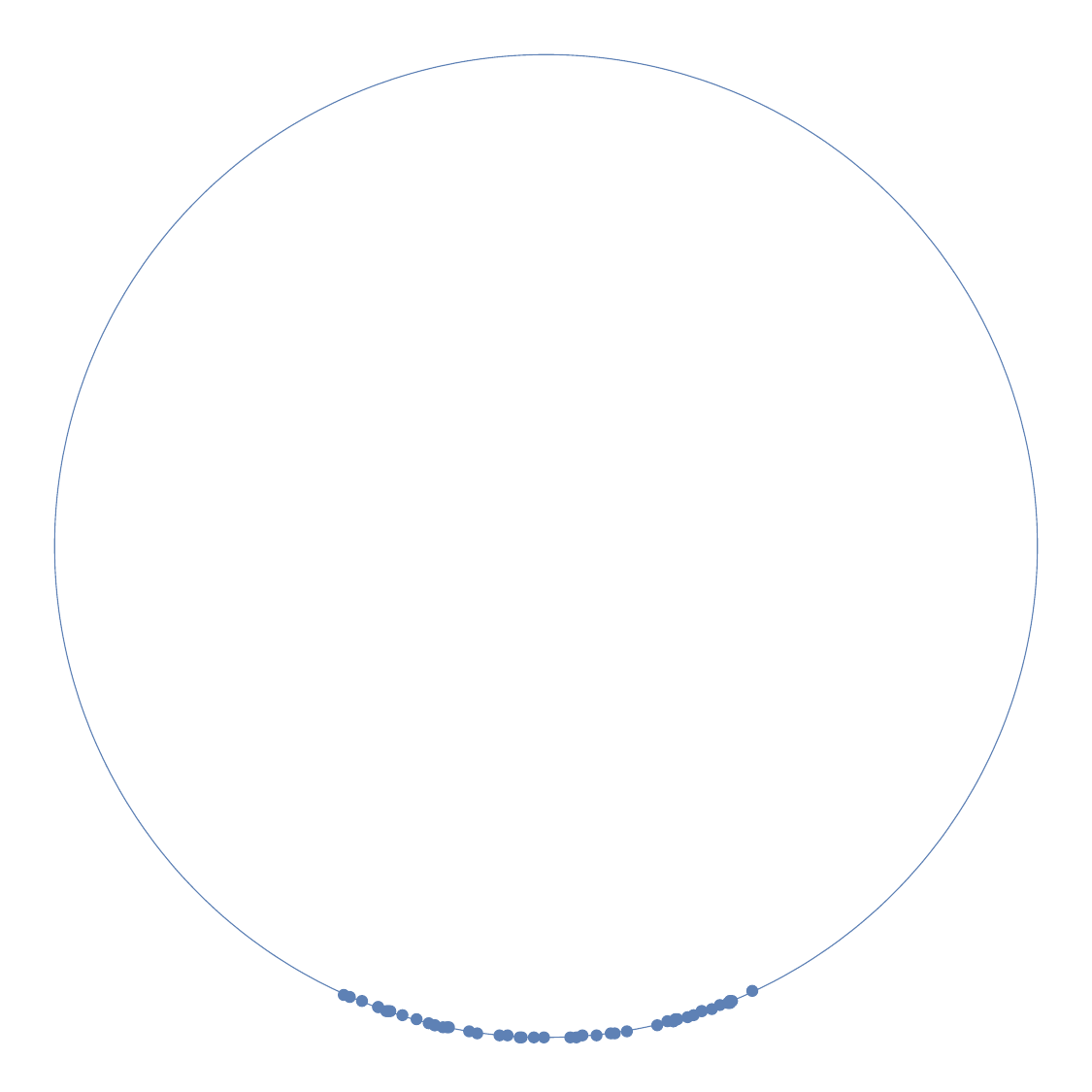}};
\draw[black,line width=0.15 mm,-<-=0.01,->-=1] ([shift=(-120:1.35cm)]0,0) arc (-120:-60:1.35cm);
\node at (0,-0.9) {$\bigO(n^{-\frac{1}{2}+\epsilon})$};
\draw[fill] (0,0) circle (0.9pt);
\node at (0.2,0) {$0$};
\end{tikzpicture}
\caption{\label{fig:mirror}Illustration of the point process \eqref{new point process intro} with $n=50$. With probability $\approx \frac{1}{2}$ all points are close to $i$ (left), and with probability $\approx \frac{1}{2}$ all points are close to  $-i$ (right).}
\end{center}
\end{figure}

Point processes with only mirror-type interactions have not been considered before to our knowledge. The main goal of this paper is to investigate the asymptotic fluctuations as $n \to \infty$ of linear statistics of the form $\sum_{j=1}^{n}g(\theta_{j})$, for fixed $\beta$ and where $g:\mathbb{R}\to \mathbb{R}$ is $2\pi$-periodic and smooth enough in neighborhoods of $\frac{\pi}{2}$ and $-\frac{\pi}{2}$ (it is already clear from Theorem \ref{thm:prob} that the regularity of $g$ outside neighborhoods of $\frac{\pi}{2}$ and $-\frac{\pi}{2}$ does not matter). One of the interesting properties of \eqref{new point process intro} is that, as shown in Theorem \ref{thm:conv in distri} below, many different types of scenarios can occur.

\medskip It is natural to expect from Theorem \ref{thm:prob} some important fluctuations in the large $n$ behavior of $\sum_{j=1}^{n}g(\theta_{j})$. Indeed, for large $n$, the average point configuration contains $\frac{n}{2}$ points near $i$ and $\frac{n}{2}$ points near $-i$, but with overwhelming probability a random point configuration contains either all $n$ points near $i$, or all $n$ points near $-i$ (see Figure \ref{fig:mirror}), and is thus ``very far" from the average. As a consequence, the empirical measure $\mu_{n} := \frac{1}{n}\sum_{j=1}^{n} \delta_{\theta_{j}}$ has no deterministic limit as $n \to \infty$. In fact, we prove in Theorem \ref{thm:conv in distri} that $\mu_{n}$ converges weakly in distribution to the random measure $\mu := B \delta_{\frac{\pi}{2}} + (1-B) \delta_{-\frac{\pi}{2}}$ where $B \sim \mathrm{Bernoulli} (\frac{1}{2})$ (i.e. $\mathbb{P}(B=0)= \mathbb{P}(B=1)=\frac{1}{2}$), namely 
\begin{align}\label{Bernoulli one cut}
\int_{(-\pi,\pi]} g(x) d\mu_{n}(x) \oset{\mathrm{law}}{\underset{n\to \infty}{\xrightarrow{\hspace*{0.85cm}}}} \int_{(-\pi,\pi]} g(x) d\mu(x).
\end{align}


\medskip We will first prove a general result about the large $n$ asymptotics of $n$-fold integrals of the form
\begin{align}\label{main integral}
I(f) = \int_{-\pi}^{\pi} \dots \int_{-\pi}^{\pi}\prod_{1 \leq j < k \leq n} |e^{i\theta_{j}}-e^{-i\theta_{k}}|^{\beta}\prod_{j=1}^{n} e^{f(\theta_{j})} d\theta_{j},
\end{align}
where $f:\mathbb{R}\to \mathbb{C}$ is regular enough and $2\pi$-periodic, see Theorem \ref{thm:main}. As a corollary, we will obtain large $n$ asymptotics for the characteristic function of $\sum_{j=1}^{n}g(\theta_{j})$ simply by considering the ratio $\mathbb{E}\big[\exp\big(it\sum_{j=1}^{n} g(\theta_{j})\big) \big] = \smash{\frac{I(itg)}{I(0)}}$, $t\in \mathbb{R}$. The large $n$ asymptotics of $Z_{n}=I(0)$ are also obtained as the special case $f=0$ of Theorem \ref{thm:main}. 

\medskip The work \cite{McKayWormald} has been the main inspiration for the present paper. In \cite{McKayWormald}, motivated by a combinatorial problem about the enumeration of regular graphs, McKay and Wormald developed a method to obtain large $n$ asymptotics of $n$-fold integrals of the form
\begin{align}\label{nfold integral McKay and Wormald}
\frac{1}{(2\pi i)^{n}} \oint \dots \oint \frac{\prod_{1 \leq j <k \leq n}(1+z_{j}z_{k})}{z_{1}^{d+1}\dots z_{n}^{d+1}}dz_{1}\ldots dz_{n},
\end{align}
where each integral is taken along a circle centered at $0$, and $d=d(n) \in \mathbb{N}$ grows with $n$ at a suitable speed. Remarkably, the method of \cite{McKayWormald} does not rely on the fact that the integrand in \eqref{nfold integral McKay and Wormald} is analytic (except for deforming the contours into circles of suitable radii), and can actually be adapted with little effort to handle integrals of the form \eqref{main integral}. 

\medskip Let $M(f):=\sup_{\theta \in (-\pi,\pi]}\re f(\theta)$. We now state our first main result.
\begin{theorem}\label{thm:main}
Fix $\beta > 0$. Let $f:\mathbb{R}\to \mathbb{C}$ be $2\pi$-periodic, bounded, and $C^{2,q}$ in neighborhoods of $\frac{\pi}{2}$ and $-\frac{\pi}{2}$, with $0<q\leq 1$. Assume also that 
\begin{align}\label{technical condition}
M(f)+\beta \log \cos(\tfrac{\pi}{16}) < \min\{\re f(\tfrac{\pi}{2}),\re f(-\tfrac{\pi}{2})\}.
\end{align}
Then, for any fixed $0<\zeta<\min\{\frac{1}{4},\frac{q}{2}\}$, as $n \to \infty$ we have
\begin{align}
I(f) & = 2^{\beta \frac{n(n-1)}{2}-\frac{1}{2}} \bigg( \frac{8\pi}{\beta n} \bigg)^{\frac{n}{2}} e^{1-\frac{1}{2 \beta}} \bigg[e^{n f(\frac{\pi}{2})}\exp\bigg( \frac{f'(\frac{\pi}{2})^{2}}{\beta} + \frac{2f''(\frac{\pi}{2})}{\beta} + \bigO(n^{-\zeta}) \bigg) \nonumber \\
& \hspace{4.07cm} + e^{n f(-\frac{\pi}{2})}\exp\bigg( \frac{f'(-\frac{\pi}{2})^{2}}{\beta} + \frac{2f''(-\frac{\pi}{2})}{\beta} + \bigO(n^{-\zeta}) \bigg) \bigg]. \label{asymptotic If}
\end{align} 
Furthermore, if $\re f \equiv 0$ and $t \in \mathbb{R}$, then the large $n$ asymptotics of $I(tf)$, which are given by \eqref{asymptotic If} with $f$ replaced by $tf$, hold uniformly for $t$ in compact subsets of $\mathbb{R}$. 
\end{theorem}
\begin{remark}
Condition \eqref{technical condition} comes from some technicalities in our analysis and can probably be weakened. For the applications on the linear statistics below, we will only use Theorem \ref{thm:main} with $f:\mathbb{R}\to i \mathbb{R}$, i.e. $\re f \equiv 0$, for which \eqref{technical condition} is automatically verified. 
\end{remark}
\begin{remark}
The leading prefactor in \eqref{asymptotic If}, namely $2^{\beta \frac{n(n-1)}{2}}$, is equal to the double product in \eqref{main integral} when $\theta_{1}=\ldots=\theta_{n}$ with $\theta_{1}\in \{-\frac{\pi}{2},\frac{\pi}{2}\}$.
\end{remark}
The following result on the characteristic function of $\sum_{j=1}^{n}g(\theta_{j})$ is a direct consequence of Theorem \ref{thm:main}.
\begin{theorem}\label{thm:moment generating function}
Fix $\beta > 0$. Let $t \in \mathbb{R}$ and let $g:\mathbb{R}\to \mathbb{R}$ be $2\pi$-periodic, bounded, and $C^{2,q}$ in neighborhoods of $\frac{\pi}{2}$ and $-\frac{\pi}{2}$, with $0<q\leq 1$. For any fixed $0<\zeta<\min\{\frac{1}{4},\frac{q}{2}\}$, as $n \to \infty$ we have
\begin{align}
\mathbb{E}\bigg[e^{it\sum_{j=1}^{n} g(\theta_{j})} \bigg] = \frac{I(itg)}{I(0)} & = \frac{e^{n i tg(\frac{\pi}{2})}}{2}\exp\bigg( -\frac{g'(\frac{\pi}{2})^{2}}{\beta}t^{2} + \frac{2g''(\frac{\pi}{2})}{\beta}it + \bigO(n^{-\zeta}) \bigg) \nonumber \\
&  + \frac{e^{n i tg(-\frac{\pi}{2})}}{2}\exp\bigg( -\frac{g'(-\frac{\pi}{2})^{2}}{\beta}t^{2} + \frac{2g''(-\frac{\pi}{2})}{\beta}it + \bigO(n^{-\zeta}) \bigg). \label{asymp for moment generating function}
\end{align}
Furthermore, the above asymptotics hold uniformly for $t$ in compact subsets of $\mathbb{R}$.
\end{theorem}
Theorem \ref{thm:moment generating function} implies, in the generic case where $g'(\frac{\pi}{2})\neq g'(-\frac{\pi}{2})$, that
\begin{align}\label{lol2}
\mathbb{E}\bigg[e^{it\sum_{j=1}^{n} g(\theta_{j}) } \bigg] & = \big(1+\bigO(n^{-\zeta})\big) \mathbb{E}[e^{n i t \big(g(\tfrac{\pi}{2})B + g(-\tfrac{\pi}{2})(1-B)\big) + it (BN_{1} + (1-B)N_{2})}]
\end{align}
holds as $n \to \infty$ uniformly for $t$ in compact subsets of $\R$, where $N_{1},N_{2},B$ are random variables independent of each other and distributed as
\begin{align*}
N_{1} \sim \mathcal{N}_{\mathbb{R}}\Big(\frac{2g''(\frac{\pi}{2})}{\beta},\frac{2g'(\frac{\pi}{2})^{2}}{\beta}\Big), \qquad N_{2} \sim \mathcal{N}_{\mathbb{R}}\Big(\frac{2g''(-\frac{\pi}{2})}{\beta},\frac{2g'(-\frac{\pi}{2})^{2}}{\beta}\Big), \qquad B \sim \mathrm{Bernoulli} \Big(\frac{1}{2}\Big),
\end{align*}
i.e. the density of $N_{1}$ is $\frac{\sqrt{\beta}}{2g'(\frac{\pi}{2})\sqrt{\pi}}\exp\big(\frac{-\beta}{4g'(\frac{\pi}{2})^{2}} (x-\frac{2g''(\frac{\pi}{2})}{\beta})^{2}\big)$. (For $g'(\frac{\pi}{2})\neq g'(-\frac{\pi}{2})$ and $t$ in compact subsets of $\R$, the expectation on the right-hand side of \eqref{lol2} stays bounded away from $0$ as $n\to\infty$; we have used this to turn the two error terms $\bigO(n^{-\zeta})$ in \eqref{asymp for moment generating function} into a single multiplicative error term in \eqref{lol2}.) Informally, one can interpret \eqref{lol2} as
\begin{align}\label{lol6}
\sum_{j=1}^{n} g(\theta_{j}) = n \Big(g(\tfrac{\pi}{2})B + g(-\tfrac{\pi}{2})(1-B)\Big) + BN_{1} + (1-B)N_{2} + o(1), \qquad \mbox{as } n \to \infty.
\end{align}
In other words, in the generic case where $g(\frac{\pi}{2})\neq g(-\frac{\pi}{2})$ and $g'(\frac{\pi}{2})\neq g'(-\frac{\pi}{2})$ hold, the leading order fluctuations of $\sum_{j=1}^{n} g(\theta_{j})$ are purely Bernoulli and of order $n$, and the subleading fluctuations are of order $1$ and of the form $BN_{1} + (1-B)N_{2}$. The leading order fluctuations are related to the global behavior of the points: $B=1$ means that all $\theta_{j}$'s are close to $\frac{\pi}{2}$ with high probability, while $B=0$ means that all $\theta_{j}$'s are close to $-\frac{\pi}{2}$ with high probability. The random variables $N_{1}$ and $N_{2}$ are, on the other hand, related to the local behavior of the $\theta_{j}$. To see this, suppose for example that $B=1$; then \eqref{lol6} can be rewritten as
\begin{align*}
\sum_{j=1}^{n} \big( g(\theta_{j})-g(\tfrac{\pi}{2}) \big) = N_{1}+o(1), \qquad \mbox{as } n \to \infty. 
\end{align*}
Thus $N_{1}$ gives information about how the $\theta_{j}$'s fluctuate around $\frac{\pi}{2}$ if $B=1$. Similarly, $N_{2}$ gives information about how the $\theta_{j}$'s fluctuate around $-\frac{\pi}{2}$ if $B=0$.

\medskip Our next theorem shows that there are also some interesting non-generic cases which produce different types of asymptotic behaviors. 
Let us define
\begin{align*}
& \nu_{1} = \frac{g(\frac{\pi}{2})+g(-\frac{\pi}{2})}{2}, & & \nu_{2} = \frac{g''(\frac{\pi}{2})+g''(-\frac{\pi}{2})}{\beta}.
\end{align*} 
We have the following result about convergence in distribution of the smooth linear statistics. 

\begin{theorem}\label{thm:conv in distri}
Fix $\beta > 0$. Let $g:\mathbb{R}\to \mathbb{R}$ be $2\pi$-periodic, bounded, and $C^{2,q}$ in neighborhoods of $\frac{\pi}{2}$ and $-\frac{\pi}{2}$, with $0<q\leq 1$. 
\begin{itemize}
\item[(a)] Define $\mu_{n} = \frac{1}{n}\sum_{j=1}^{n} \delta_{\theta_{j}}$. We have
\begin{align*}
\int_{(-\pi,\pi]} g(x) d\mu_{n}(x) \oset{\mathrm{law}}{\underset{n\to \infty}{\xrightarrow{\hspace*{0.85cm}}}} \int_{(-\pi,\pi]} g(x) d\mu(x),
\end{align*}
where $\mu := B \delta_{\frac{\pi}{2}} + (1-B) \delta_{-\frac{\pi}{2}}$. Equivalently, 
\begin{align*}
n^{-1}\bigg(\sum_{j=1}^{n} g(\theta_{j})-\nu_{1}n\bigg)  \oset{\mathrm{law}}{\underset{n\to \infty}{\xrightarrow{\hspace*{0.85cm}}}} \frac{g(\frac{\pi}{2})-g(-\frac{\pi}{2})}{2}(2B-1).
\end{align*}
\item[(b)] If $g(\frac{\pi}{2}) = g(-\frac{\pi}{2})$ and $g'(\frac{\pi}{2}) \neq 0 \neq g'(-\frac{\pi}{2})$, then
\begin{align*}
\sum_{j=1}^{n} g(\theta_{j})-\nu_{1}n \oset{\mathrm{law}}{\underset{n\to \infty}{\xrightarrow{\hspace*{0.85cm}}}} BN_{1} + (1-B)N_{2}.
\end{align*}
In particular, if $g(\frac{\pi}{2}) = g(-\frac{\pi}{2})$, $g'(\frac{\pi}{2}) = g'(-\frac{\pi}{2}) \neq 0$ and $g''(\frac{\pi}{2})=g''(-\frac{\pi}{2})$, then $N_{1}$ and $N_{2}$ are equal in distribution and
\begin{align*}
\sum_{j=1}^{n} g(\theta_{j})-\nu_{1}n \oset{\mathrm{law}}{\underset{n\to \infty}{\xrightarrow{\hspace*{0.85cm}}}} N_{1}.
\end{align*}
\item[(c)] If $g(\frac{\pi}{2}) = g(-\frac{\pi}{2})$, $g'(\frac{\pi}{2}) \neq 0$, and $g'(-\frac{\pi}{2})=0$, then
\begin{align*}
\sum_{j=1}^{n} g(\theta_{j})-\nu_{1}n \oset{\mathrm{law}}{\underset{n\to \infty}{\xrightarrow{\hspace*{0.85cm}}}} BN_{1} + (1-B)\frac{2g''(-\frac{\pi}{2})}{\beta}. 
\end{align*}
Similarly, if $g(\frac{\pi}{2}) = g(-\frac{\pi}{2})$, $g'(-\frac{\pi}{2}) \neq 0$, and $g'(\frac{\pi}{2})=0$, then
\begin{align*}
\sum_{j=1}^{n} g(\theta_{j})-\nu_{1}n \oset{\mathrm{law}}{\underset{n\to \infty}{\xrightarrow{\hspace*{0.85cm}}}} B\frac{2g''(\frac{\pi}{2})}{\beta} + (1-B)N_{2}.
\end{align*}
\item[(d)] If $g(\frac{\pi}{2}) = g(-\frac{\pi}{2})$ and $g'(\frac{\pi}{2}) = 0 = g'(-\frac{\pi}{2})$, then
\begin{align*}
\sum_{j=1}^{n} g(\theta_{j})-(\nu_{1}n + \nu_{2}) \oset{\mathrm{law}}{\underset{n\to \infty}{\xrightarrow{\hspace*{0.85cm}}}} \frac{g''(\frac{\pi}{2})-g''(-\frac{\pi}{2})}{\beta}(2B-1).
\end{align*}
\end{itemize}
\end{theorem}
\begin{remark}
Each random variable in Theorem \ref{thm:conv in distri} has a variance that increases as $\beta^{-1}$ increases. This is consistent with the expectation that as $\beta^{-1}$ increases, the random point configurations of \eqref{new point process intro} should become less localized around $(i,\ldots,i)$ and $(-i,\ldots,-i)$.
\end{remark}
\begin{remark}\label{remark:other rv}
If $g(\frac{\pi}{2}) = g(-\frac{\pi}{2})$, $g'(\frac{\pi}{2}) = 0 = g'(-\frac{\pi}{2})$ and $g''(\frac{\pi}{2}) = g''(-\frac{\pi}{2})$, then $\sum_{j=1}^{n} g(\theta_{j})-(\nu_{1}n + \nu_{2})$ is typically of order $\bigO(n^{-\zeta})$ and our result \eqref{asymp for moment generating function} is not precise enough to understand the fluctuations in this situation. 
\end{remark}
\begin{proof}[Proof of Theorem \ref{thm:conv in distri}]
Recall that \eqref{asymp for moment generating function} holds uniformly for $t$ in compact subsets of $\mathbb{R}$. Hence, using \eqref{asymp for moment generating function} but with $t$ replaced by $s/n$, $s \in \mathbb{R}$ fixed, we obtain
\begin{align*}
\mathbb{E}\bigg[e^{isn^{-1}(\sum_{j=1}^{n} g(\theta_{j}) - n\nu_{1})} \bigg] & = \frac{e^{i s\frac{g(\frac{\pi}{2})-g(-\frac{\pi}{2})}{2}}}{2}\exp\big( \bigO(n^{-\zeta}) \big) + \frac{e^{i s \frac{g(-\frac{\pi}{2})-g(\frac{\pi}{2})}{2}}}{2}\exp\big( \bigO(n^{-\zeta}) \big) \\
& = \mathbb{E}[e^{is \frac{g(\frac{\pi}{2})-g(-\frac{\pi}{2})}{2}(2B-1)}] + \bigO(n^{-\zeta})
\end{align*}
as $n \to \infty$. Claim (a) now directly follows from L\'{e}vy's continuity theorem. Let us now consider the case $g(\frac{\pi}{2}) = g(-\frac{\pi}{2})$. Using again \eqref{asymp for moment generating function}, but now with $t\in \mathbb{R}$ fixed, we obtain
\begin{align*}
\mathbb{E}\bigg[e^{it(\sum_{j=1}^{n} g(\theta_{j}) - n\nu_{1})} \bigg] = \frac{1}{2}e^{-\frac{g'(\frac{\pi}{2})^{2}}{\beta}t^{2} + \frac{2g''(\frac{\pi}{2})}{\beta}it + \bigO(n^{-\zeta}) } + \frac{1}{2} e^{ -\frac{g'(-\frac{\pi}{2})^{2}}{\beta}t^{2} + \frac{2g''(-\frac{\pi}{2})}{\beta}it + \bigO(n^{-\zeta}) }
\end{align*}
as $n \to \infty$. Since 
\begin{align*}
& \mathbb{E}[e^{it (BN_{1} + (1-B)N_{2})}] = \frac{1}{2}\mathbb{E}[e^{it N_{1}}] + \frac{1}{2}\mathbb{E}[e^{it N_{2}}] = \frac{1}{2}e^{-\frac{g'(\frac{\pi}{2})^{2}}{\beta}t^{2} + \frac{2g''(\frac{\pi}{2})}{\beta}it} + \frac{1}{2} e^{ -\frac{g'(-\frac{\pi}{2})^{2}}{\beta}t^{2} + \frac{2g''(-\frac{\pi}{2})}{\beta}it}, 
\end{align*}
claim (b) now also directly follows from L\'{e}vy's continuity theorem. The proofs of claims (c) and (d) are similar to that of claim (b).
\end{proof}

\vspace{-0.5cm}\paragraph{Comparison with other point processes.}
We are not aware of an earlier work on a point process with only mirror-type interactions such as \eqref{new point process intro}. There is however a vast literature on point processes with different types of interactions, such as
\begin{itemize}
\item[(a)] \vspace{-0.15cm} the C$\beta$E, given by $\sim\prod_{j < k}|e^{i\theta_{j}}-e^{i\theta_{k}}|^{\beta}\prod_{j=1}^{n} d\theta_{j}$,
\item[(b)] \vspace{-0.15cm} point processes on the unit circle with Riesz pairwise interactions,
\item[(c)] \vspace{-0.15cm} the point process on the unit circle given by $\sim\prod_{j < k} |e^{i\theta_{j}}-e^{i\theta_{k}}|^{\beta} |e^{i\theta_{j}}-e^{-i\theta_{k}}|^{\beta} \prod_{j=1}^{n} d\theta_{j}$,
\item[(d)] \vspace{-0.2cm} the two-dimensional point process $\sim \hspace{-0.08cm} \prod_{j < k}\hspace{-0.08cm}|z_{j}-z_{k}|^{2}|z_{j}-\overline{z}_{k}|^{2}\hspace{-0.05cm}\prod_{j=1}^{n}\hspace{-0.08cm}|z_{j}-\overline{z}_{j}|^{2}e^{-N|z_{j}|^{2}}\hspace{-0.06cm}d^{2}z_{j}$,
\item[(e)] \vspace{-0.15cm} the point process on the unit circle given by $\sim\prod_{j < k} |e^{i\theta_{j}}+e^{i\theta_{k}}|^{\beta}\prod_{j=1}^{n} d\theta_{j}$. 
\end{itemize}
Other examples of point processes can be found in e.g. \cite{For, LS2017}. The five examples listed above share at least one common feature with \eqref{new point process intro}: (a), (b), (c) and (e) are point processes defined on the unit circle, and (c) and (d) are point processes involving the image points reflected across the real line. In sharp contrast with \eqref{new point process intro}, the C$\beta$E and the circular Riesz gas favor the configurations with equispaced points on the unit circle and the associated smooth linear statistics always have Gaussian fluctuations (except for constant test functions, in which case there are obviously no fluctuations), see e.g. \cite{Jo1988} for the C$\beta$E and \cite{Boursier} for the Riesz gas. Example (c) is discussed in \cite[Section 2.9]{For} (see also \cite{KN2004}) for its connection to random matrix theory. Here the mirror-type interactions $\prod_{j < k}|e^{i\theta_{j}}-e^{-i\theta_{k}}|^{\beta}$ are damped by the pairwise repulsion $\prod_{j < k}|e^{i\theta_{j}}-e^{i\theta_{k}}|^{\beta}$, and just like (a) and (b), the limiting empirical measure $\frac{1}{n}\sum_{j=1}^{n} \delta_{\theta_{j}}$ of point process (c) is the uniform measure on $(-\pi,\pi]$, see \cite[Proposition 3.6.3 and Exercise 4.1.1]{For}. Example (d) is an integrable Pfaffian point process on the plane introduced by Ginibre \cite{Ginibre} and later generalized in several works, see e.g. \cite{ABK2022} and the review \cite{BF2022}. Here too, the mirror-type interactions $\prod_{j < k}|z_{j}-\overline{z}_{k}|^{2}\prod_{j=1}^{n}|z_{j}-\overline{z}_{j}|^{2}$ are damped by the pairwise repulsion $\prod_{j < k}|z_{j}-z_{k}|^{2}$. The overall effect is that the repulsion between the points and the real axis is only visible on local scales, see e.g. \cite[Figure 1(b)]{ABK2022} (this is in sharp contrast with \eqref{new point process intro}, see Theorem \ref{thm:prob}).

\medskip As mentioned at the beginning of the introduction, the point process (e) with antipodal interactions is considered in \cite{C ReflectionPoint}. In contrast with (a), (b), (c) and (d), and just like \eqref{new point process intro}, the empirical measure $\mu_{n}^{e}:=\frac{1}{n}\sum_{j=1}^{n}\delta_{\theta_{j}}$ of (e) has no deterministic limit as $n \to \infty$.  In fact, it is proved in \cite{C ReflectionPoint} that if $g:\mathbb{R}\to \mathbb{R}$ is $2\pi$-periodic and $C^{1,q}$ for some $q\in (0,1]$, then
\begin{align}\label{Uniform one cut}
\int_{(-\pi,\pi]} g(x) d\mu_{n}^{e}(x) \oset{\mathrm{law}}{\underset{n\to \infty}{\xrightarrow{\hspace*{0.85cm}}}} \int_{(-\pi,\pi]} g(x) d\mu^{e}(x),
\end{align}
where $\mu^{e} = \delta_{U}$ and $U \sim \mathrm{Uniform}(-\pi,\pi]$. The convergence in \eqref{Uniform one cut} implies that the leading order fluctuations of the smooth linear statistics $\sum_{j=1}^{n}g(\theta_{j})$ of (e) are of order $n$ and given by $g(U)n$. It is also proved in \cite{C ReflectionPoint} that, as $n\to\infty$,
\begin{align}\label{antipodal asymptotics}
\sum_{j=1}^{n} g(\theta_{j}) = n \, g(U) + \sqrt{n} \; N_{U} + o(\sqrt{n}), \qquad \mbox{where } N_{U}\sim \mathcal{N}_{\R}(0,\tfrac{4g'(U)^{2}}{\beta}).
\end{align}
For the precise sense in which \eqref{antipodal asymptotics} holds, see \cite{C ReflectionPoint}. It is interesting to compare \eqref{antipodal asymptotics} with \eqref{lol6}. In particular, for the point process (e), the subleading fluctuations of $\sum_{j=1}^{n}g(\theta_{j})$ are of order $\sqrt{n}$ and given by $N_{U}\sqrt{n}$, i.e. by a Gaussian random variable with a random variance. In contrast, the smooth linear statistics of \eqref{new point process intro} have no subleading fluctuations of order $\sqrt{n}$, see \eqref{lol6}.

\paragraph{Concluding remarks and open problems.} In this paper, we investigated the smooth linear statistics of \eqref{new point process intro}. More questions can be asked about this point process. For example:
\begin{itemize}
\item \textbf{Counting statistics.} What is the asymptotic behavior of $\sum_{j=1}^{n}g(\theta_{j})$ if $g$ is not smooth in neighborhoods of $\frac{\pi}{2}$ and $-\frac{\pi}{2}$? Let $\mathcal{I}\subset (-\pi,\pi]$ be an interval (possibly depending on $n$). A particular test function $g$ of interest is
\begin{align*}
g(\theta) = \begin{cases}
1, & \mbox{if } \theta \in \mathcal{I}, \\
0, & \mbox{if } \theta \in (-\pi,\pi] \setminus \mathcal{I}.
\end{cases}
\end{align*}
In this case the random variable $\sum_{j=1}^{n}g(\theta_{j})$ counts the number of points lying in $\mathcal{I}$. 
\item \textbf{Gap probabilities.} What is the asymptotic behavior of $\mathbb{P}(\#\{\theta_{j} \in \mathcal{I}\} = 0)$? This gap probability can be rewritten as $\mathbb{E}\big[\exp\big(\sum_{j=1}^{n} g(\theta_{j}) \big) \big]$ with
\begin{align*}
g(\theta) = \begin{cases}
-\infty, & \mbox{if } \theta \in \mathcal{I}, \\
0, & \mbox{if } \theta \in (-\pi,\pi] \setminus \mathcal{I}.
\end{cases}
\end{align*}
\item \textbf{Third order asymptotics for the smooth statistics.} Theorem \ref{thm:moment generating function} provides second order asymptotics for $\mathbb{E}\big[\exp(it\sum_{j=1}^{n} g(\theta_{j})) \big]$ when $g$ is $C^{2,q}$ in neighborhoods of $\frac{\pi}{2}$ and $-\frac{\pi}{2}$. If we assume $g \in C^{3,q}$, what is the third term? This is relevant in view of Remark \ref{remark:other rv}.
\item \textbf{Small and large $\beta$.} The results of this paper are valid for fixed $\beta>0$. What if $\beta$ depends on $n$ such that either $\beta \to 0$ (more randomness) or $\beta \to \infty$ (less randomness)? The asymptotic formula \eqref{asymp for moment generating function} suggests that a critical transition occurs when $\beta \asymp n^{-1}$.
\end{itemize}
All of the above questions are probably difficult and will require new techniques.

\paragraph{Outline.} The general strategy to prove Theorems \ref{thm:prob} and \ref{thm:main} relies on some ideas from \cite{McKayWormald}, and can roughly be summarized as follows:
\begin{enumerate}
\item\label{step1} The first step is to show that the main contribution to the $n$-fold integral $I(f)$ (defined in \eqref{main integral}) comes from small neighborhoods of size $n^{-\frac{1}{2}+\epsilon}$ around $(\frac{\pi}{2},\ldots,\frac{\pi}{2})$ and $(-\frac{\pi}{2},\ldots,-\frac{\pi}{2})$. More precisely, we will show that 
\begin{align}\label{lol9}
& I(f) = J_{3}(0) \big( 1+ \bigO(e^{-cn^{2\epsilon}}) \big) + \tilde{J}_{3}(0) \big( 1+ \bigO(e^{-cn^{2\epsilon}}) \big),
\end{align}
as $n\to \infty$, where
\begin{align*}
& J_{3}(0) := \int_{\frac{\pi}{2}-n^{\epsilon-\frac{1}{2}}}^{\frac{\pi}{2}+n^{\epsilon-\frac{1}{2}}} \dots \int_{\frac{\pi}{2}-n^{\epsilon-\frac{1}{2}}}^{\frac{\pi}{2}+n^{\epsilon-\frac{1}{2}}} \prod_{1 \leq j < k \leq n} |e^{i\theta_{j}}-e^{-i\theta_{k}}|^{\beta}\prod_{j=1}^{n} e^{f(\theta_{j})} d\theta_{j},  \\
& \tilde{J}_{3}(0) := \int_{-\frac{\pi}{2}-n^{\epsilon-\frac{1}{2}}}^{-\frac{\pi}{2}+n^{\epsilon-\frac{1}{2}}} \dots \int_{-\frac{\pi}{2}-n^{\epsilon-\frac{1}{2}}}^{-\frac{\pi}{2}+n^{\epsilon-\frac{1}{2}}} \prod_{1 \leq j < k \leq n} |e^{i\theta_{j}}-e^{-i\theta_{k}}|^{\beta}\prod_{j=1}^{n} e^{f(\theta_{j})} d\theta_{j}. \nonumber
\end{align*}
\item\label{step2} The second step is to expand the integrand of $J_{3}(0)$ around $(\frac{\pi}{2},\ldots,\frac{\pi}{2})$, see \eqref{lol7} (and similarly for $\tilde{J}_{3}(0)$). One is then left to analyze an integral of the form
\begin{align}
& \int_{U_{n}(n^{-\frac{1}{2}+\epsilon})} \exp \bigg( \alpha_{1} \sum_{j<k}(\eta_{j}+\eta_{k})^{2} + \alpha_{2} \sum_{j<k}(\eta_{j}+\eta_{k})^{4} + \bigO\bigg(\sum_{j<k}(\eta_{j}+\eta_{k})^{6}\bigg) \nonumber \\
& \hspace{5.05cm} + \alpha_{3} \sum_{j=1}^{n}\eta_{j} + \alpha_{4} \sum_{j=1}^{n} \eta_{j}^{2} + \bigO\bigg(\sum_{j=1}^{n} \eta_{j}^{2+q}\bigg) \bigg) \prod_{j=1}^{n} d\eta_{j}, \label{lol12}
\end{align}
where $U_{n}(t) := \{\boldsymbol{x} \in \mathbb{R}^{n}: |x_{i}| \leq t, i=1,\ldots,n\}$. Since the $\eta_{j}$ are small, the leading order term as $n\to\infty$ in the above integral comes from the linear and quadratic terms in the exponential. Note however that the quadratic terms do not decouple, because of the double sum $\sum_{j<k}(\eta_{j}+\eta_{k})^{2}$. 
\item\label{step3} The third step overcomes this difficulty by applying a change of variables $\boldsymbol{\eta}=T\boldsymbol{y}$, using an operator $T$ which decouples the quadratic terms. We then obtain an integral of the form
\begin{align*}
\int_{T^{-1}U_{n}(n^{-\frac{1}{2}+\epsilon})} h(\boldsymbol{y})\prod_{j=1}^{n} dy_{j},
\end{align*}
where $h$ is as in Lemma \ref{eta to y transformation}. 
\item The fourth step consists in finding the large $n$ asymptotics of 
\begin{align}\label{lol8}
\int_{U_{n}(n^{-\frac{1}{2}+\epsilon})} h(\boldsymbol{y})\prod_{j=1}^{n} dy_{j}.
\end{align}
This is done in Lemma \ref{eta to y transformation}.
\item In the fifth step, we use the inclusions 
\begin{align}\label{inclusion}
U_{n}(\tfrac{n^{-\frac{1}{2}+\epsilon}}{1+\gamma}) \subseteq T^{-1}U_{n}(n^{-\frac{1}{2}+\epsilon}) \subseteq U_{n}(\tfrac{n^{-\frac{1}{2}+\epsilon}}{1-\gamma}),
\end{align}
where $\gamma \approx 1-\frac{1}{\sqrt{2}}$, to show that
\begin{align*}
\int_{T^{-1}U_{n}(n^{-\frac{1}{2}+\epsilon})\setminus U_{n}(n^{-\frac{1}{2}+\epsilon})} h(\boldsymbol{y})\prod_{j=1}^{n} dy_{j}
\end{align*}
is negligible in comparison to \eqref{lol8}.
\item The final step consists in substituting the asymptotics of $J_{3}(0)$ and $\tilde{J}_{3}(0)$ in \eqref{lol9}.
\end{enumerate}

\begin{remark}\label{remark:comparison antipodal} (comparison with \cite{C ReflectionPoint})

\noindent We already discussed above the differences in the asymptotic behavior of the linear statistics associated with the two point processes \eqref{new point process intro} and \eqref{AntipodalType} (see \eqref{lol6} and \eqref{antipodal asymptotics}). In this remark, we focus the discussion on the technical differences at the level of the proof.

\noindent In \cite{C ReflectionPoint}, the smooth linear statistics of the point process \eqref{AntipodalType} with antipodal interactions are analyzed using techniques from \cite{McKay, IsaevMcKay}. The general strategy outlined above for \eqref{new point process intro} also applies for \eqref{AntipodalType}, but there are five important differences, which we mention here:
\begin{enumerate}
\item[(i)] In this paper, the smooth linear statistics of \eqref{new point process intro} are analyzed by establishing a precise asymptotic formula for the characteristic function
\begin{align}\label{lol10}
\mathbb{E}\bigg[e^{it\sum_{j=1}^{n} g(\theta_{j})} \bigg],
\end{align}
see Theorem \ref{thm:moment generating function}. Instead of that, one could also have considered the moment generating function
\begin{align}\label{lol11}
\mathbb{E}\bigg[e^{t\sum_{j=1}^{n} g(\theta_{j})} \bigg].
\end{align}
However, it turns out that obtaining rigorously the large $n$ asymptotics of \eqref{lol11} is much more challenging than that of \eqref{lol10}, especially for the subleading term. One way to see this is to note that condition \eqref{technical condition} holds for $f=itg$ and $t\in \R$, but may fail for $f=tg$ if $|t|$ is too large. 

When analyzing the point process \eqref{AntipodalType}, one faces the important extra difficulty that many error terms in the analysis become out of control if the integrand is complex-valued. Hence, in \cite{C ReflectionPoint}, the analysis of the smooth linear statistics proceeds via the moment generating function instead of the characteristic function, and the proof boils down to obtaining the large $n$ asymptotics of
\begin{align}\label{main integral anti}
\int_{-\pi}^{\pi} \dots \int_{-\pi}^{\pi}\prod_{1 \leq j < k \leq n} |e^{i\theta_{j}}+e^{i\theta_{k}}|^{\beta}\prod_{j=1}^{n} e^{\frac{t}{\sqrt{n}}g(\theta_{j})} d\theta_{j},
\end{align}
where $g$ is real-valued.
\item[(ii)] The first step in \cite{C ReflectionPoint} consists in showing that the main contribution to \eqref{main integral anti} comes from a small neighborhood of the set $\{(\theta,\ldots,\theta):\theta \in (-\pi,\pi]\}$. Since this set is of dimension one, this step is harder than its counterpart \ref{step1}.
\item[(iii)] The second step consists in expanding the integrand, in a similar way as in \ref{step2}. The integral that must then be analyzed is of the form
\begin{align}
I_{1} & = 2^{\beta \frac{n(n-1)}{2}}\exp \Big( \bigO(n^{-\frac{q}{2}+(1+q) \epsilon}+n^{-1+6\epsilon}) \Big) \int_{-\pi}^{\pi} e^{t\sqrt{n} \,  g(\theta_{n})} I_{1}'(\theta_{n}) d\theta_{n}, \label{lol8 2} \\
I_{1}'(\theta_{n}) & := \int_{\theta_{n}+U_{n-1}(n^{-\frac{1}{2}+\epsilon})} \exp \bigg( - \frac{\beta}{8}\sum_{1 \leq j<k \leq n}(\theta_{j}-\theta_{k})^{2} - \frac{\beta}{192}\sum_{1 \leq j<k \leq n}(\theta_{j}-\theta_{k})^{4} \nonumber \\
& \hspace{7cm} + \frac{tg'(\theta_{n})}{\sqrt{n}}\sum_{j=1}^{n-1}(\theta_{j}-\theta_{n}) \bigg) \prod_{j=1}^{n-1} d\theta_{j}, \nonumber
\end{align}
see also \cite[eq (3.4)]{C ReflectionPoint}. Note that the error terms in \eqref{lol8 2} appear in front of the integral. This is needed for the rest of the proof in \cite{C ReflectionPoint}, and can only be justified if the integrand is real-valued. This also explains why the moment generating function is considered in \cite{C ReflectionPoint} instead of the characteristic function.

This difficulty is not present in this paper when analyzing the point process with mirror-type interactions: in \eqref{lol12}, it is not problematic that the error terms stay inside the integral. This allows us to study the characteristic function, which involves a complex-valued integrand. 
\item[(iv)] The third step consists in applying a change of variables $\boldsymbol{y}=\tilde{T}\boldsymbol{\theta}$ that decouples the quadratic term $\sum_{1 \leq j<k \leq n}(\theta_{j}-\theta_{k})^{2}$. This step is similar in spirit to \ref{step3}, with however yet another important difference: 
\begin{itemize}
\item when analyzing the point process with mirror-type interactions, the operator $T$ of step \ref{step3} is well-behaved, in the sense that $\det T = 1-\gamma \approx \frac{1}{\sqrt{2}}$;
\item when analyzing the point process with antipodal interactions, the operator $\tilde{T}$ becomes singular as $n\to + \infty$, in the sense that $\det \tilde{T} = n^{-\frac{1}{2}}$.
\end{itemize}
This difference is also related to the fact that the set $\{(\theta,\ldots,\theta):\theta \in (-\pi,\pi]\}$ is of dimension one, i.e. that the ``saddle points" of \eqref{main integral anti} are not isolated from one another.
\item[(v)] After applying the change of variables $\boldsymbol{y}=\tilde{T}\boldsymbol{\theta}$, we are left to evaluate an integral over the domain $V_{1}=\tilde{T}(U_{n-1}(n^{-\frac{1}{2}+\epsilon}))$. For the mirror-type point process, the set $T^{-1}U_{n}(n^{-\frac{1}{2}+\epsilon})$ obeys the two inclusions \eqref{inclusion}. In contrast, the domain $V_{1}$ is more complicated and only satisfies $V_{1} \subset U_{n-1}(2n^{-\frac{1}{2}+\epsilon})$. This produces some technical challenges.
\end{enumerate}
For all the above reasons, analyzing the point process \eqref{AntipodalType} with antipodal interactions is more challenging than analyzing the point process \eqref{new point process intro}. 
\end{remark}

\section{Preliminaries}\label{section:prelim}

In this section, we introduce some notation and record some results from \cite{McKayWormald}. These results will be used in Section \ref{section:proof} to obtain large $n$ asymptotics for $I(f)$.
\begin{lemma}\label{lemma:exp bound for cos}(Special case of \cite[Lemma 1]{McKayWormald}.) For all $x\in \mathbb{R}$,
\begin{align*}
|\tfrac{1+e^{ix}}{2}| = (\tfrac{1+\cos x}{2})^{\frac{1}{2}} = |\cos \tfrac{x}{2}| \leq \exp (-\tfrac{x^{2}}{8}+\tfrac{x^{4}	}{96}).
\end{align*}
\end{lemma}
\begin{lemma}\label{lemma:inequality 2 and 4}(\cite[Eq (3.3)]{McKayWormald})
Let $\ell \in \mathbb{N}_{>0}$. For all $x_{1},\ldots,x_{\ell}\in \mathbb{R}$,
\begin{align*}
\sum_{1\leq j < k \leq \ell}(x_{j}+x_{k})^{2} \geq (\ell-2) \sum_{j=1}^{\ell}x_{j}^{2}, \qquad \sum_{1 \leq j < k \leq \ell}(x_{j}+x_{k})^{4} \leq 8(\ell-1) \sum_{j=1}^{\ell}x_{j}^{4}.
\end{align*}
\end{lemma}
Following \cite[Section 2]{McKayWormald}, we also introduce the following quantities:
\begin{align*}
& \gamma = 1 - \sqrt{\frac{n-2}{2(n-1)}}, \quad J_{n} = \mbox{the } n \times n \mbox{ matrix of all ones}, \quad I_{n} = \mbox{the } n \times n \mbox{ identity matrix}, \\
& T = I_{n} - \gamma J_{n}/n, \qquad \boldsymbol{y},\boldsymbol{\eta} \in \mathbb{R}^{n}, \qquad \boldsymbol{\eta} = T \boldsymbol{y}, \qquad \mu_{k} = \sum_{j=1}^{n}y_{j}^{k} \quad \mbox{for } k \geq 0, \\
& U_{n}(t) = \{\boldsymbol{x} \in \mathbb{R}^{n}: |x_{i}| \leq t, i=1,\ldots,n\} \quad \mbox{for } t \geq 0.
\end{align*}
\begin{lemma}\label{eta to y transformation}(\cite[Lemma 2]{McKayWormald})
\begin{align*}
& \mbox{(a)} \;\; \sum_{j=1}^{n} \eta_{j} = (1-\gamma)\mu_{1}, \qquad \sum_{j=1}^{n} \eta_{j}^{2} = \mu_{2}-\gamma(2-\gamma)\mu_{1}^{2}/n, \quad \sum_{1 \leq j < k \leq n}(\eta_{j}+\eta_{k})^{2} = (n-2)\mu_{2}, \\
& \;\;\; \sum_{1 \leq j < k \leq n}(\eta_{j}+\eta_{k})^{4} = (n-8)\mu_{4} + 3\mu_{2}^{2} + \big( 4(1-2\gamma)+32\gamma/n \big)\mu_{1}\mu_{3} - \big( 24\gamma(1-\gamma)/n + 48 \gamma^{2}/n^{2} \big)\mu_{1}^{2}\mu_{2} \\
& \hspace{3.2cm} + \big( 8\gamma^{2}(1-\gamma)(3-\gamma)/n^{2} + 8\gamma^{3}(4-\gamma)/n^{3} \big)\mu_{1}^{4}. \\
& \mbox{(b) } \det(I_{n}-sJ_{n}/n)=1-s \mbox{ for any } s. \\
& \mbox{(c) For any } t \geq 0, TU_{n}(t) \subseteq (1+\gamma)U_{n}(t) \mbox{ and } T^{-1}U_{n}(t) \subseteq (1-\gamma)^{-1}U_{n}(t).
\end{align*}
\end{lemma}
The following lemma is a minor extension of \cite[Lemma 3]{McKayWormald} (see also \cite[Section 4]{IsaevMcKay} for similar theorems).
\begin{lemma}\label{lemma:ABCD}
Let $a=9$, $\epsilon=\epsilon(n)$ and $\epsilon'=\epsilon'(n)$ be such that $0<\epsilon'<2\epsilon<\frac{1}{a}$. Let $A=A(n)$ be a bounded complex-valued function such that $\im A=\bigO(n^{-1})$ and $\re A \geq n^{-\epsilon'}$ for sufficiently large $n$. Let $B=B(n),C=C(n),\ldots,K=K(n)$ be complex valued functions such that the ratios $B/A,C/A,\ldots, K/A$ are bounded. Suppose that $\delta>0$, $0 < \Delta < \frac{1}{4}-\frac{1}{2}\epsilon$, and that
\begin{align*}
h(\boldsymbol{y}) = \exp \big( & -An\mu_{2} + Bn\mu_{3} + C\mu_{1}\mu_{2} + D \mu_{1}^{3}/n + E n \mu_{4} + F\mu_{2}^{2} \\
& + G \mu_{1}\mu_{3} + H \mu_{1}^{2}\mu_{2}/n + I \mu_{1}^{4}/n^{2} + J \mu_{1} + K \mu_{1}^{2}/n + \bigO(n^{-\delta}) \big)
\end{align*}
is integrable for $\boldsymbol{y} \in U_{n}(n^{-\frac{1}{2}+\epsilon})$. Then, provided the error term converges to zero,
\begin{align*}
\int_{U_{n}(n^{-\frac{1}{2}+\epsilon})} h(\boldsymbol{y})\prod_{j=1}^{n} dy_{j} = \bigg( \frac{\pi}{An} \bigg)^{\frac{n}{2}} \exp \bigg( & \frac{J^{2}}{4A} + \frac{3E+F+(C+3B)J}{4A^{2}}+\frac{15B^{2}+6BC+C^{2}}{16A^{3}} \\
& + \bigO\big( (n^{-\frac{1}{2}+a\epsilon}+n^{-\delta})Z + n^{-1+12\epsilon} + A^{-1} n^{-\Delta} \big) \bigg),
\end{align*}
where 
\begin{align*}
Z = \exp \bigg( \frac{15 \, \im(B)^{2}+6\, \im(B)(\im(C) + 2 \,\re(A) \, \im(J))+(\im(C)+2\,\re(A)\, \im(J))^{2}}{16\,\re(A)^{3}} \bigg).
\end{align*}
Furthermore, if $D(n) \equiv 0$, then the statement holds with $a=7$.\footnote{Even for $K\equiv 0$ the statement only holds for $a=9$ (or $a=7$ if $D(n) \equiv 0$). This lemma is stated in \cite{McKayWormald} for $K \equiv 0$ and $a=6$. However, it seems to us that there is a small typo in \cite{McKayWormald} and that \cite[Eq (2.2)]{McKayWormald} only holds for $\eta = \frac{3}{2}-9\epsilon$ (or $\eta = \frac{3}{2}-7\epsilon$ if $D(n) \equiv 0$).}
\end{lemma}
\begin{proof}
The proof for $K(n)\equiv 0$ is done in \cite[Proof of Lemma 3]{McKayWormald}. The case of non-zero $K$ only requires to modify $\psi_{m}(\boldsymbol{y})$ in \cite[Proof of Lemma 3]{McKayWormald} into
\begin{align*}
\psi_{m}(\boldsymbol{y}) = \exp \big( & -An\mu_{2} + E n \mu_{4} + F\mu_{2}^{2} + Bn\hat{\mu}_{3} + C\hat{\mu}_{1}\mu_{2} + J \hat{\mu}_{1} + D \hat{\mu}_{1}^{3}/n + G \hat{\mu}_{1}\hat{\mu}_{3}  \\
&  + H \hat{\mu}_{1}^{2}\mu_{2}/n + I \hat{\mu}_{1}^{4}/n^{2}  + K \hat{\mu}_{1}^{2}/n + \tfrac{1}{2}B^{2}n^{2}\check{\mu}_{6} + \tfrac{1}{2}(C\mu_{2}+J)^{2}\check{\mu}_{2} + B(C\mu_{2}+J)n\check{\mu}_{4} \\
& + \tfrac{9}{2}D^{2}\check{\mu}_{2}\hat{\mu}_{1}^{4}/n^{2} + (3BD\check{\mu}_{4} + 3(C\mu_{2}+J)D\check{\mu}_{2}/n)\hat{\mu}_{1}^{2}  \big).
\end{align*}
Also, in \cite{McKayWormald} this lemma is stated for real-valued $A$, but the extension to complex-valued $A$ with $\im A = \bigO(n^{-1})$ is straightforward.
\end{proof}

The following lemma was proved in \cite[Top of p. 572]{McKayWormald} for $k=0$. In the present paper, we will also use it for $k=1$.
\begin{lemma}\label{lemma:exp int over 2t}
If $k\in \N$ and $\delta \in (0, \frac{1}{4})$ are fixed and $t\in [\frac{m^{-\frac{1}{2}+\delta}}{2},\frac{\pi}{8}]$, then as $m \to \infty$,
\begin{align*}
& \int_{-2t}^{2t} x^{2k}\exp \big( -mx^{2}+\tfrac{2}{3}m(1+o(1))x^{4} \big)dx = \frac{\Gamma(k+\frac{1}{2})}{m^{k+\frac{1}{2}}} \big( 1+\bigO(m^{-1+4\delta}) \big),
\end{align*}
where $\Gamma(z) := \int_{0}^{+\infty} t^{z-1}e^{-t}dt$ is the standard Gamma function (above, the $o(1)$-term is assumed to be uniform for $x\in [-2t,2t]$). In particular,  for $k=0$ and $k=1$, this gives
\begin{align*}
& \int_{-2t}^{2t} \exp \big( -mx^{2}+\tfrac{2}{3}m(1+o(1))x^{4} \big)dx = \sqrt{\frac{\pi}{m}} \big( 1+\bigO(m^{-1+4\delta}) \big), \\
& \int_{-2t}^{2t} x^{2}\exp \big( -mx^{2}+\tfrac{2}{3}m(1+o(1))x^{4} \big)dx = \frac{\sqrt{\pi}}{2m^{3/2}} \big( 1+\bigO(m^{-1+4\delta}) \big).
\end{align*}
\end{lemma}
\begin{proof}
Since the maximum of $\exp \big( -mx^{2}+\tfrac{2}{3}m(1+o(1)) x^{4} + \frac{2k}{m}\log |x| \big)$ over $[-\frac{\pi}{4},\frac{\pi}{4}]\setminus [-m^{-\frac{1}{2}+\delta},m^{-\frac{1}{2}+\delta}]$ is attained at $x=\pm m^{-\frac{1}{2}+\delta}$ for all $m$ sufficiently large, we have
\begin{multline*}
\int_{[-2t,2t]\setminus [-m^{-\frac{1}{2}+\delta},m^{-\frac{1}{2}+\delta}]} x^{2k}\exp \big( -mx^{2}+\tfrac{2}{3}m(1+o(1))x^{4} \big)dx \\
\leq 4t \; m^{-k+2\delta k}e^{-m^{2\delta}+\bigO(m^{-1+4\delta})} = \bigO(e^{-\frac{1}{2}m^{2\delta}})
\end{multline*}
for all large enough $m$. On the other, by Taylor expanding in a neighborhood of $0$, we obtain
\begin{align*}
& \int_{-m^{-\frac{1}{2}+\delta}}^{m^{-\frac{1}{2}+\delta}} x^{2k}\exp \big( -mx^{2}+\tfrac{2}{3}m(1+o(1))x^{4} \big)dx = \big( 1+\bigO(m^{-1+4\delta}) \big) \int_{-m^{-\frac{1}{2}+\delta}}^{m^{-\frac{1}{2}+\delta}} x^{2k}e^{-mx^{2}}dx \\
& = \big( 1+\bigO(m^{-1+4\delta}) \big) \frac{1}{m^{k+\frac{1}{2}}} \int_{-m^{\delta}}^{m^{\delta}} y^{2k}e^{-y^{2}}dy = \big( 1+\bigO(m^{-1+4\delta}) \big) \frac{1}{m^{k+\frac{1}{2}}} \int_{-\infty}^{+\infty} y^{2k}e^{-y^{2}}dy
\end{align*}
as $m\to \infty$. Since $\int_{-\infty}^{+\infty} y^{2k}e^{-y^{2}}dy = \Gamma(k+\frac{1}{2})$, the claim follows.
\end{proof}

\section{Proof of Theorems \ref{thm:prob} and \ref{thm:main}}\label{section:proof}
We divide the proof into two parts: we will first prove \eqref{asymptotic If} and Theorem \ref{thm:prob}, and the text written below \eqref{asymptotic If} about the dependence of the error terms in $t$ when $\re f \equiv 0$ will be proved afterwards.

\subsection{Proof of \eqref{asymptotic If} and of Theorem \ref{thm:prob}}\label{subsection:general f}

The proof closely follows the ideas of \cite[Proof of Theorem 1]{McKayWormald}. For convenience, we first make the change of variables $\eta_{j} = \theta_{j}-\frac{\pi}{2}$ in \eqref{main integral}; this yields
\begin{align}\label{If after change of variables}
I(f) = \int_{-\pi}^{\pi} \dots \int_{-\pi}^{\pi}\prod_{1 \leq j < k \leq n} |e^{i\eta_{j}}+e^{-i\eta_{k}}|^{\beta}\prod_{j=1}^{n} e^{\mathsf{f}(\eta_{j})} d\eta_{j},
\end{align}
where $\mathsf{f}(\eta) := f(\eta + \frac{\pi}{2})$. We first show that the main contribution to $I(f)$ comes from the point configurations for which either all the $e^{i\eta_{j}}$ are close to $1$, or all the $e^{i\eta_{j}}$ are close to $-1$. Let $\tau=\frac{\pi}{8}$ and fix $\epsilon \in (0,\frac{1}{14})$. If $\mathsf{f}$ is not identically zero, then we also assume that $\epsilon < \frac{q}{2(2+q)}$. Consider the partition $(-\pi,\pi]^{n} = \mathcal{J}_{1} \sqcup \mathcal{J}_{1}^{c}$, with
\begin{align*}
& \mathcal{J}_{1} := \bigg\{\boldsymbol{\eta}=(\eta_{1},\dots,\eta_{n})\in (-\pi,\pi]^{n}: n_{0}n_{2} \geq n^{1+\epsilon} \mbox{ or } \binom{n_{1}}{2} \geq n^{1+\epsilon} \mbox{ or } \binom{n_{3}}{2} \geq n^{1+\epsilon} \bigg\},
\end{align*}
and where $n_{0}=n_{0}(\boldsymbol{\eta}),n_{1}=n_{1}(\boldsymbol{\eta}),n_{2}=n_{2}(\boldsymbol{\eta}),n_{3}=n_{3}(\boldsymbol{\eta})$ are the numbers of $\eta_{j}$ in the regions $[-\tau,\tau]$, $(\tau,\pi-\tau)$, $[\pi-\tau,\pi]\cup (-\pi,-\pi+\tau]$ and $(-\pi+\tau,-\tau)$, respectively. Define
\begin{align*}
J_{1} = \int_{\mathcal{J}_{1}} \prod_{1 \leq j < k \leq n} |e^{i\eta_{j}}+e^{-i\eta_{k}}|^{\beta}\prod_{j=1}^{n} e^{\mathsf{f}(\eta_{j})} d\eta_{j}.
\end{align*}
Using $|e^{i\eta_{j}}+e^{-i\eta_{k}}|^{\beta}=2^{\beta}\big|\cos \tfrac{\eta_{j}+\eta_{k}}{2}\big|^{\beta}$, we get
\begin{align}
|J_{1}| & \leq e^{n M(f)}(2\pi)^{n} 2^{\beta\frac{n(n-1)}{2}} \bigg[ (\cos \tau)^{\beta\frac{n_{1}(n_{1}-1)}{2}} + (\cos \tau)^{\beta\frac{n_{3}(n_{3}-1)}{2}} + (\cos \tau)^{\beta n_{0}n_{2}} \bigg] \nonumber \\
& \leq 3 \, e^{n M(f)}(2\pi)^{n} 2^{\beta\frac{n(n-1)}{2}} (\cos \tau)^{\beta n^{1+\epsilon}}. \label{bound on J1}
\end{align}
It remains to estimate the integral over $\mathcal{J}_{1}^{c}$, for which we have $n_{1} = \bigO(n^{\frac{1+\epsilon}{2}})$, $n_{3} = \bigO(n^{\frac{1+\epsilon}{2}})$, and either $n_{0} = \bigO(n^{\epsilon})$ or $n_{2} = \bigO(n^{\epsilon})$. For sufficiently large $n$, we can write $\mathcal{J}_{1}^{c} = \mathcal{J}_{2} \sqcup \tilde{\mathcal{J}}_{2}$, where 
\begin{align*}
& \mathcal{J}_{2} = \bigg\{\boldsymbol{\eta}\in \mathcal{J}_{1}^{c}:  n_{1} \leq 2n^{\frac{1+\epsilon}{2}} \mbox{ and } n_{3} \leq 2n^{\frac{1+\epsilon}{2}} \mbox{ and } n_{2} \leq  n^{2\epsilon} \bigg\}, \\
& \tilde{\mathcal{J}}_{2} = \bigg\{\boldsymbol{\eta}\in \mathcal{J}_{1}^{c}:  n_{1} \leq 2n^{\frac{1+\epsilon}{2}} \mbox{ and } n_{3} \leq 2n^{\frac{1+\epsilon}{2}} \mbox{ and } n_{0} \leq  n^{2\epsilon} \bigg\}.
\end{align*}
We first consider the $n$-fold integral over $\mathcal{J}_{2}$. Define $S_{0}=S_{0}(\boldsymbol{\eta})$, $S_{1}=S_{1}(\boldsymbol{\eta})$ and $S_{2}=S_{2}(\boldsymbol{\eta})$ by
\begin{align*}
& S_{0} = \{j : |\eta_{j}| \leq \tau\}, & & S_{1} = \{j : \tau < |\eta_{j}| \leq 2\tau\}, & & S_{2} = \{j: 2\tau<|\eta_{j}| \leq \pi\},
\end{align*}
and let $s_{0} = \# S_{0}$, $s_{1} = \# S_{1}$ and $s_{2} = \# S_{2}$. For $\boldsymbol{\eta} \in \mathcal{J}_{2}$, we note that $s_{0} = n_{0} \geq n-5n^{\frac{1+\epsilon}{2}}$ and $s_{1}+s_{2} = n_{1}+n_{2}+n_{3} \leq 5n^{\frac{1+\epsilon}{2}}$. Moreover, we have
\begin{align}\label{bounds}
\big|\cos \tfrac{\eta_{j}+\eta_{k}}{2}\big|^{\beta} \leq \begin{cases}
\exp(-\frac{\beta}{8}(\eta_{j}+\eta_{k})^{2} + \frac{\beta}{96}(\eta_{j}+\eta_{k})^{4}), & \mbox{if } j,k \in S_{0}\cup S_{1}, \\
(\cos \frac{\tau}{2})^{\beta}, & \mbox{if } j \in S_{0}, \; k \in S_{2}, \\
1, & 
\end{cases}
\end{align}
where for the top inequality we have used Lemma \ref{lemma:exp bound for cos}. Let $\alpha := -\beta\log \cos \frac{\tau}{2}$. Using \eqref{bounds} and Lemma \ref{lemma:inequality 2 and 4}, we infer that the modulus of the integrand in \eqref{If after change of variables} is bounded above for $\boldsymbol{\eta}\in\mathcal{J}_{2}$ by
\begin{align}
& 2^{\beta\frac{n(n-1)}{2}} \exp \bigg( -\frac{\beta}{8}\sum_{\substack{1 \leq j<k \leq n\\j,k \in S_{0}\cup S_{1}}}(\eta_{j}+\eta_{k})^{2} + \frac{\beta}{96}\sum_{\substack{1 \leq j<k \leq n\\j,k \in S_{0}\cup S_{1}}}(\eta_{j}+\eta_{k})^{4} - \alpha s_{0} s_{2} + \sum_{j=1}^{n}\re \mathsf{f}(\eta_{j}) \bigg) \label{a nice inequality} \\
& \leq e^{n M(f)}2^{\beta\frac{n(n-1)}{2}} \exp \bigg( -\frac{\beta}{8}(n-s_{2}-2)\sum_{j \in S_{0}\cup S_{1}} \eta_{j}^{2} + \frac{\beta}{12}(n-s_{2}-1)\sum_{j \in S_{0}\cup S_{1}}\eta_{j}^{4} - \alpha s_{2} \big( n - 5n^{\frac{1+\epsilon}{2}} \big) \bigg). \nonumber
\end{align}
Let $J_{2}(m_{2})$ be the contribution to \eqref{If after change of variables} from $\{\boldsymbol{\eta} \in \mathcal{J}_{2}:  s_{2}(\boldsymbol{\eta})=m_{2} \}$. Using the above inequality and Lemma \ref{lemma:exp int over 2t} (with $\delta=\frac{1}{6}$ and $k=0$), for all sufficiently large $n$ and $m_{2} \leq 5n^{\frac{1+\epsilon}{2}}$ we obtain
\begin{align*}
& |J_{2}(m_{2})| \leq \bigg| \binom{n}{m_{2}} \int_{|\eta_{1}|,\ldots,|\eta_{m_{2}}| \in (2\tau,\pi]}\int_{\substack{|\eta_{m_{2}+1}|,\ldots,|\eta_{n}|\leq 2\tau \\ s_{0}(\boldsymbol{\eta}) \geq n-5n^{\frac{1+\epsilon}{2}}}} \prod_{1 \leq j < k \leq n} |e^{i\eta_{j}}+e^{-i\eta_{k}}|^{\beta}\prod_{j=1}^{n} e^{\mathsf{f}(\eta_{j})} d\eta_{j} \bigg| \\
&  \leq e^{n M(f)} 2^{\beta\frac{n(n-1)}{2}} (2\pi-4\tau)^{m_{2}} e^{- \alpha m_{2} ( n - 5n^{\frac{1+\epsilon}{2}} )}  \binom{n}{m_{2}} \sqrt{\frac{\pi}{\frac{\beta}{8}(n-m_{2}-2)}}^{n-m_{2}} \big( 1+\bigO(n^{-\frac{1}{3}}) \big)^{n-m_{2}} \\
& \leq e^{n M(f)} 2^{\beta\frac{n(n-1)}{2}} (2\pi)^{m_{2}} e^{- \alpha m_{2} ( n - 5n^{\frac{1+\epsilon}{2}} )} n^{m_{2}} \bigg( \frac{8\pi}{\beta n} \bigg)^{\frac{n}{2}}  e^{\bigO(n^{2/3})}.
\end{align*}
Hence, 
\begin{align}\label{sum J2 bound}
\sum_{m_{2}=1}^{5n^{\frac{1+\epsilon}{2}}} |J_{2}(m_{2})| \leq e^{n (M(f)-\alpha)} 2^{\beta\frac{n(n-1)}{2}} \bigg( \frac{8\pi}{\beta n} \bigg)^{\frac{n}{2}} e^{\bigO( \frac{n}{\log n})}.
\end{align}
We now turn to the analysis of $J_{2}(0)$. For this, define $\tilde{S}_{0}=\tilde{S}_{0}(\boldsymbol{\eta})$ and $\tilde{S}_{1}=\tilde{S}_{1}(\boldsymbol{\eta})$ by
\begin{align*}
& \tilde{S}_{0} = \{j : |\eta_{j}| \leq n^{-\frac{1}{2}+\epsilon}\}, & & \tilde{S}_{1} = \{j : n^{-\frac{1}{2}+\epsilon} < |\eta_{j}| \leq 2\tau\},
\end{align*}
and let $\tilde{s}_{0} = \# \tilde{S}_{0}$ and $\tilde{s}_{1} = \# \tilde{S}_{1}$. Define also $J_{3}(\tilde{m}_{1})$ to be the contribution to \eqref{If after change of variables} from $\{\boldsymbol{\eta} \in (-\pi,\pi]^{n}: \tilde{s}_{0}(\boldsymbol{\eta}) = n-\tilde{m}_{1} \mbox{ and } \tilde{s}_{1}(\boldsymbol{\eta}) = \tilde{m}_{1} \}$, and note that $J_{2}(0) = \sum_{\tilde{m}_{1}=0}^{n}J_{3}(\tilde{m}_{1})$. In the same way as we proved \eqref{a nice inequality} (but with $s_{2}=0$), we note that the modulus of the integrand in \eqref{If after change of variables} is bounded above by
\begin{align}\label{lol1}
& 2^{\beta\frac{n(n-1)}{2}} e^{\tilde{s}_{1} M(f)} \exp \bigg( -\frac{\beta}{8}(n-2)\sum_{j =1}^{n} \eta_{j}^{2} + \frac{\beta}{12}(n-1)\sum_{j =1}^{n}\eta_{j}^{4} + \sum_{j\in \tilde{S}_{0}}\re \mathsf{f}(\eta_{j}) \bigg).
\end{align}
Using \eqref{lol1} and Lemma \ref{lemma:exp int over 2t} with $\delta = \frac{\epsilon}{4}$, $k=0$ and $k=2$ (see also \eqref{lol13} and \eqref{lol14} below), we find 
\begin{align*}
& |J_{3}(\tilde{m}_{1})|  = \bigg| \binom{n}{\tilde{m}_{1}} \int_{|\eta_{1}|,\ldots,|\eta_{\tilde{m}_{1}}|\in (n^{-\frac{1}{2}+\epsilon},2\tau]}\int_{|\eta_{\tilde{m}_{1}+1}|,\ldots,|\eta_{n}|\leq n^{-\frac{1}{2}+\epsilon}} \prod_{1 \leq j < k \leq n} |e^{i\eta_{j}}+e^{-i\eta_{k}}|^{\beta}\prod_{j=1}^{n} e^{\mathsf{f}(\eta_{j})} d\eta_{j} \bigg| \\
&  \leq  2^{\beta\frac{n(n-1)}{2}} n^{\tilde{m}_{1}} \bigg( \int_{-n^{- \frac{1}{2}+\epsilon}}^{n^{- \frac{1}{2}+\epsilon}} e^{\re \mathsf{f}(x)}\exp\big(-\tfrac{\beta}{8}(n-2) x^{2} + \tfrac{\beta}{12}(n-1) x^{4}\big)dx \bigg)^{n-\tilde{m}_{1}} \\
& \hspace{1.8cm} \times e^{\tilde{m}_{1} M(f)} \bigg( 2 \int_{n^{- \frac{1}{2}+\epsilon}}^{2\tau} \exp\big(-\tfrac{\beta}{8}(n-2) x^{2} + \tfrac{\beta}{12}(n-1) x^{4}\big)dx \bigg)^{\tilde{m}_{1}} \\
& \leq 2^{\beta\frac{n(n-1)}{2}} n^{\tilde{m}_{1}} e^{(n-\tilde{m}_{1})\re \mathsf{f}(0)} \bigg( \frac{8 \pi}{\beta (n-2)} \bigg)^{\frac{n-\tilde{m}_{1}}{2}} e^{\bigO(n^{\epsilon})} \times e^{\tilde{m}_{1}M(f)}2^{\tilde{m}_{1}} \exp \Big( -\tfrac{\beta}{16}n^{2\epsilon}\tilde{m}_{1} \Big) \\
& \leq e^{n \, \re \mathsf{f}(0)} 2^{\beta\frac{n(n-1)}{2}} \bigg( \frac{8 \pi}{\beta n} \bigg)^{\frac{n}{2}} \exp \Big( -\tfrac{\beta}{16}n^{2\epsilon}\tilde{m}_{1} + \bigO(n^{\epsilon}) + \bigO(\tilde{m}_{1} \log n) \Big)
\end{align*}
for all sufficiently large $n$. For the second inequality above, we have used (recall that $\mathsf{f}$ is $C^{2,q}$ in a neighborhood of $0$)
\begin{align}
& \bigg( \int_{-n^{- \frac{1}{2}+\epsilon}}^{n^{- \frac{1}{2}+\epsilon}} e^{\re \mathsf{f}(x)}\exp\big(-\tfrac{\beta}{8}(n-2) x^{2} + \tfrac{\beta}{12}(n-1) x^{4}\big)dx \bigg)^{n-\tilde{m}_{1}} \nonumber \\
& \leq \bigg( e^{\re \mathsf{f}(0)}\int_{-n^{- \frac{1}{2}+\epsilon}}^{n^{- \frac{1}{2}+\epsilon}} (1+Cx^{2}) \exp\big(-\tfrac{\beta}{8}(n-2) x^{2} + \tfrac{\beta}{12}(n-1) x^{4}\big)dx \bigg)^{n-\tilde{m}_{1}} \nonumber \\
& = \bigg( e^{\re \mathsf{f}(0)} \frac{\sqrt{8\pi}}{\sqrt{\beta(n-2)}} \big( 1+\bigO(n^{-1+\epsilon}) \big) \bigg)^{n-\tilde{m}_{1}} = e^{(n-\tilde{m}_{1})\re \mathsf{f}(0)} \bigg( \frac{8\pi}{\beta(n-2)} \bigg)^{\frac{n-\tilde{m}_{1}}{2}}e^{\bigO(n^{\epsilon})}. \label{lol13}
\end{align}
We have also used that the maximum of $\exp\big(-\tfrac{\beta}{8}(n-2) x^{2} + \tfrac{\beta}{12}(n-1) x^{4}\big)$ over the bounded interval $[n^{- \frac{1}{2}+\epsilon},2\tau]$ is attained at $x=n^{- \frac{1}{2}+\epsilon}$ for all $n$ sufficiently large to conclude that
\begin{align}\label{lol14}
\int_{n^{- \frac{1}{2}+\epsilon}}^{2\tau} \exp\big(-\tfrac{\beta}{8}(n-2) x^{2} + \tfrac{\beta}{12}(n-1) x^{4}\big)dx \leq \exp \Big( -\tfrac{\beta}{16}n^{2\epsilon} \Big)
\end{align}
holds for all large $n$.

Hence, for some $c_{3}>0$ and all large enough $n$,
\begin{align}\label{bound on sum J3}
\sum_{\tilde{m}_{1}=1}^{n}|J_{3}(\tilde{m}_{1})| \leq e^{n \, \re \mathsf{f}(0)} 2^{\beta\frac{n(n-1)}{2}} \bigg( \frac{8 \pi}{\beta n} \bigg)^{\frac{n}{2}} e^{-c_{3}n^{2\epsilon}}.
\end{align}
Finally, we turn to the analysis of $J_{3}(0)$. Since $\mathsf{f}$ is $C^{2,q}$ in a neighborhood of $0$, as $x \to 0$ we have
\begin{align*}
& \log\big[ 2^{\beta}\,|\hspace{-0.04cm}\cos \tfrac{x}{2}|^{\beta} \big] = \beta \log 2 - \frac{\beta}{8}x^{2} - \frac{\beta}{192}x^{4} + \bigO(x^{6}), \\
& \mathsf{f}(x) = \mathsf{f}(0) + \mathsf{f}'(0)x + \frac{1}{2}\mathsf{f}''(0) x^{2} + \bigO(x^{2+q}),
\end{align*}
and thus
\begin{align}
J_{3}(0) & = 2^{\beta \frac{n(n-1)}{2}}e^{n \, \mathsf{f}(0)} \int_{U_{n}(n^{-\frac{1}{2}+\epsilon})} \exp \bigg( - \frac{\beta}{8}\sum_{j<k}(\eta_{j}+\eta_{k})^{2} - \frac{\beta}{192}\sum_{j<k}(\eta_{j}+\eta_{k})^{4} + \bigO\bigg(\sum_{j<k}(\eta_{j}+\eta_{k})^{6}\bigg) \nonumber \\
& \hspace{4.75cm} + \mathsf{f}'(0) \sum_{j=1}^{n}\eta_{j} + \frac{\mathsf{f}''(0)}{2} \sum_{j=1}^{n} \eta_{j}^{2} + \bigO\bigg(\sum_{j=1}^{n} \eta_{j}^{2+q}\bigg) \bigg) \prod_{j=1}^{n} d\eta_{j}. \label{lol7}
\end{align}
For $\boldsymbol{\eta} \in U_{n}(n^{-\frac{1}{2}+\epsilon})$, 
\begin{align*}
& \bigO\bigg(\sum_{j<k}(\eta_{j}+\eta_{k})^{6}\bigg) = \bigO(n^{-1+6\epsilon}), \qquad \bigO\bigg(\sum_{j=1}^{n} \eta_{j}^{2+q}\bigg) = \bigO(n^{-\frac{q}{2}+(2+q)\epsilon}).
\end{align*}
Since $\epsilon \in (0,\frac{1}{14})$ is fixed, we have $n^{-\frac{q}{2}+(2+q)\epsilon} + n^{-1+6\epsilon} = \bigO(n^{-\frac{q}{2}+(2+q)\epsilon})$. Hence, applying the transformation $\boldsymbol{\eta}=T\boldsymbol{y}$ of Section \ref{section:prelim}, and using Lemma \ref{eta to y transformation} (a) and (b) (using in particular that $\det T = 1-\gamma = \frac{1}{\sqrt{2}}(1+\bigO(n^{-1})$), we obtain
\begin{align}
J_{3}(0) & = 2^{\beta \frac{n(n-1)}{2}}\frac{e^{n \mathsf{f}(0)}}{\sqrt{2}} \big( 1+\bigO(n^{-1}) \big) \int_{T^{-1}U_{n}(n^{-\frac{1}{2}+\epsilon})} \hspace{-0.15cm} \exp \bigg\{ \hspace{-0.15cm} - \hspace{-0.05cm} \frac{\beta}{8} (n-2)\mu_{2} - \frac{\beta}{192}\bigg[ (n-8)\mu_{4}   \nonumber \\
& + \bigg( 4(1-2\gamma)+\frac{32\gamma}{n} \bigg) \mu_{1}\mu_{3} + 3\mu_{2}^{2} - \bigg( \frac{24\gamma(1-\gamma)}{n} + \frac{48\gamma^{2}}{n^{2}} \bigg) \mu_{1}^{2}\mu_{2} \nonumber \\
& + \bigg( 8\gamma^{2}(1-\gamma)(3-\gamma)\frac{1}{n^{2}}+8\gamma^{3}(4-\gamma)\frac{1}{n^{3}} \bigg) \mu_{1}^{4} \bigg] + \mathsf{f}'(0) (1-\gamma)\mu_{1} \nonumber \\
&  + \frac{\mathsf{f}''(0)}{2}\bigg( \mu_{2} - \gamma(2-\gamma)\frac{\mu_{1}^{2}}{n} \bigg) + \bigO\big(n^{-\frac{q}{2}+(2+q)\epsilon}\big) \bigg\} \prod_{j=1}^{n} dy_{j}. \label{J30 isaev}
\end{align}
Note that the above $\bigO\big(n^{-\frac{q}{2}+(2+q)\epsilon}\big)$ term can be replaced by $\bigO\big(n^{-1+6\epsilon}\big)$ if $f \equiv 0$. If $f \not\equiv 0$, then $\bigO\big(n^{-\frac{q}{2}+(2+q)\epsilon}\big)$ decays since we assume that $\epsilon < \frac{q}{2(2+q)}$.

By Lemma \ref{eta to y transformation} (c), $U_{n}(\frac{n^{-\frac{1}{2}+\epsilon}}{1+\gamma}) \subseteq T^{-1}U_{n}(n^{-\frac{1}{2}+\epsilon}) \subseteq U_{n}(\frac{n^{-\frac{1}{2}+\epsilon}}{1-\gamma})$. Let $\mathcal{G}(\boldsymbol{y})$ be the argument of the exponential in \eqref{J30 isaev}, and let $\epsilon_{n}:=-\log(1-\gamma)/\log n$. We have
\begin{align}
J_{3}(0) & = 2^{\beta \frac{n(n-1)}{2}}\frac{e^{n \mathsf{f}(0)}}{\sqrt{2}} \int_{U_{n}(n^{-\frac{1}{2}+\epsilon-\epsilon_{n}})} \exp (\mathcal{G}(\boldsymbol{y}))d\boldsymbol{y} \nonumber \\
& + 2^{\beta \frac{n(n-1)}{2}}\frac{e^{n \mathsf{f}(0)}}{\sqrt{2}} \int_{T^{-1}U_{n}(n^{-\frac{1}{2}+\epsilon}) \setminus U_{n}(n^{-\frac{1}{2}+\epsilon-\epsilon_{n}})} \exp (\mathcal{G}(\boldsymbol{y}))d\boldsymbol{y}. \label{lol3}
\end{align}
For the first integral over $U_{n}(n^{-\frac{1}{2}+\epsilon-\epsilon_{n}})$, since $\epsilon \in (0,\frac{1}{14})$, we can apply Lemma \ref{lemma:ABCD} with $\delta = \frac{q}{2}-(2+q)\epsilon >0$ and
\begin{align}
& A=\frac{\beta}{8}\frac{n-2}{n}-\frac{\mathsf{f}''(0)}{2n}, \quad B=0, \quad C=0, \quad D=0, \quad E = -\frac{\beta}{192}\frac{n-8}{n}, \quad F=-\frac{\beta}{64}, \nonumber \\
& G=-\frac{\beta}{48}\bigg( 1-2\gamma + \frac{8\gamma}{n} \bigg), \quad H=\frac{\beta}{8}\bigg( \gamma(1-\gamma)+\frac{2\gamma^{2}}{n} \bigg), \quad I = - \frac{\beta}{24}\bigg( \gamma^{2}(1-\gamma)(3-\gamma) + \frac{\gamma^{3}(4-\gamma)}{n} \bigg), \nonumber \\
& J=\mathsf{f}'(0)(1-\gamma), \quad K = - \frac{\mathsf{f}''(0)}{2}\gamma(2-\gamma), \label{coeff ABCDEF in proof}
\end{align}
to get
\begin{multline}\label{lol4}
2^{\beta \frac{n(n-1)}{2}}\frac{e^{n \mathsf{f}(0)}}{\sqrt{2}} \int_{U_{n}(n^{-\frac{1}{2}+\epsilon-\epsilon_{n}})} \exp (\mathcal{G}(\boldsymbol{y}))d\boldsymbol{y} \\
= 2^{\beta \frac{n(n-1)}{2}}\frac{e^{n \mathsf{f}(0)}}{\sqrt{2}} \bigg( \frac{8\pi}{\beta n} \bigg)^{\frac{n}{2}} \exp\bigg( 1-\frac{1}{2\beta}+\frac{\mathsf{f}'(0)^{2}}{\beta} + \frac{2\mathsf{f}''(0)}{\beta} + \bigO(n^{-\zeta'}) \bigg),
\end{multline}
for any fixed $\zeta' < \min\{\frac{q}{2}-(2+q)\epsilon,\frac{1}{4}-\frac{1}{2}\epsilon\}$. The integral over $T^{-1}U_{n}(n^{-\frac{1}{2}+\epsilon}) \setminus U_{n}(n^{-\frac{1}{2}+\epsilon-\epsilon_{n}})$ in \eqref{lol3} can be estimated as follows:
\begin{align}
& \bigg|2^{\beta \frac{n(n-1)}{2}}\frac{e^{n \mathsf{f}(0)}}{\sqrt{2}} \int_{T^{-1}U_{n}(n^{-\frac{1}{2}+\epsilon}) \setminus U_{n}(n^{-\frac{1}{2}+\epsilon-\epsilon_{n}})} \exp (\mathcal{G}(\boldsymbol{y}))d\boldsymbol{y}\bigg| \nonumber \\
& \leq 2^{\beta \frac{n(n-1)}{2}}\frac{e^{n \re \mathsf{f}(0)}}{\sqrt{2}} \int_{U_{n}(n^{-\frac{1}{2}+\epsilon+\epsilon_{n}}) \setminus U_{n}(n^{-\frac{1}{2}+\epsilon-\epsilon_{n}})} \exp (\re \mathcal{G}(\boldsymbol{y}))d\boldsymbol{y} \nonumber \\
& \leq 2^{\beta \frac{n(n-1)}{2}}\frac{e^{n \re \mathsf{f}(0)}}{\sqrt{2}} \bigg( \frac{8\pi}{\beta n} \bigg)^{\frac{n}{2}} \bigO(n^{-\zeta}), \label{lol5}
\end{align}
for any fixed $\zeta < \min\{\frac{q}{2}-(2+q)\epsilon,\frac{1}{4}-\frac{1}{2}\epsilon\}$, where for the last inequality we have used twice Lemma \ref{lemma:ABCD} with
\begin{align}
& A=\frac{\beta}{8}\frac{n-2}{n}-\frac{\re \mathsf{f}''(0)}{2n}, \quad B=0, \quad C=0, \quad D=0, \quad E = -\frac{\beta}{192}\frac{n-8}{n}, \quad F=-\frac{\beta}{64}, \nonumber \\
& G=-\frac{\beta}{48}\bigg( 1-2\gamma + \frac{8\gamma}{n} \bigg), \quad H=\frac{\beta}{8}\bigg( \gamma(1-\gamma)+\frac{2\gamma^{2}}{n} \bigg), \quad I = - \frac{\beta}{24}\bigg( \gamma^{2}(1-\gamma)(3-\gamma) + \frac{\gamma^{3}(4-\gamma)}{n} \bigg), \nonumber \\
& J=(1-\gamma)\re \mathsf{f}'(0), \quad K = - \frac{\re \mathsf{f}''(0)}{2}\gamma(2-\gamma). \label{coeff ABCDEF in proof Re}
\end{align}
By combining \eqref{lol3}, \eqref{lol4} and \eqref{lol5}, we obtain
\begin{align}\label{J30}
J_{3}(0) = 2^{\beta \frac{n(n-1)}{2}}\frac{e^{n \mathsf{f}(0)}}{\sqrt{2}} \bigg( \frac{8\pi}{\beta n} \bigg)^{\frac{n}{2}} \exp\bigg( 1-\frac{1}{2\beta}+\frac{\mathsf{f}'(0)^{2}}{\beta} + \frac{2\mathsf{f}''(0)}{\beta} + \bigO(n^{-\zeta}) \bigg).
\end{align}
Using now \eqref{sum J2 bound}, \eqref{bound on sum J3}, \eqref{J30} and \eqref{technical condition}, we conclude that 
\begin{align*}
\int \ldots \int_{\mathcal{J}_{2}} |e^{i\eta_{j}}+e^{-i\eta_{k}}|^{\beta}\prod_{j=1}^{n} e^{\mathsf{f}(\eta_{j})} d\eta_{j} = J_{3}(0) \big( 1+ \bigO(e^{-cn^{2\epsilon}}) \big),
\end{align*}
for some $c>0$. Similarly, reducing $c>0$ if necessary, we find
\begin{align*}
\int \ldots \int_{\tilde{\mathcal{J}}_{2}} |e^{i\eta_{j}}+e^{-i\eta_{k}}|^{\beta}\prod_{j=1}^{n} e^{\mathsf{f}(\eta_{j})} d\eta_{j} = \tilde{J}_{3}(0) \big( 1+ \bigO(e^{-cn^{2\epsilon}}) \big),
\end{align*}
where $\tilde{J}_{3}(0)$ satisfies
\begin{align}\label{J30 bis}
\tilde{J}_{3}(0) = 2^{\beta \frac{n(n-1)}{2}}\frac{e^{n\mathsf{f}(\pi)}}{\sqrt{2}} \bigg( \frac{8\pi}{\beta n} \bigg)^{\frac{n}{2}} \exp\bigg( 1-\frac{1}{2\beta}+\frac{\mathsf{f}'(\pi)^{2}}{\beta} + \frac{2\mathsf{f}''(\pi)}{\beta} + \bigO(n^{-\zeta}) \bigg).
\end{align}
Hence, by \eqref{bound on J1}, \eqref{sum J2 bound}, \eqref{bound on sum J3}, \eqref{J30} and \eqref{J30 bis}, we have
\begin{align}
I(f) & = J_{3}(0) \big( 1+ \bigO(e^{-cn^{2\epsilon}}) \big) + \tilde{J}_{3}(0) \big( 1+ \bigO(e^{-cn^{2\epsilon}}) \big) \label{prob interpretation}\\
& = 2^{\beta \frac{n(n-1)}{2}-\frac{1}{2}} \bigg( \frac{8\pi}{\beta n} \bigg)^{\frac{n}{2}} \bigg[e^{n\mathsf{f}(0)}\exp\bigg( 1-\frac{1}{2\beta}+\frac{\mathsf{f}'(0)^{2}}{\beta} + \frac{2\mathsf{f}''(0)}{\beta} + \bigO(n^{-\zeta}) \bigg) \nonumber \\
& \hspace{3.22cm} + e^{n\mathsf{f}(\pi)}\exp\bigg( 1-\frac{1}{2\beta}+\frac{\mathsf{f}'(\pi)^{2}}{\beta} + \frac{2\mathsf{f}''(\pi)}{\beta} + \bigO(n^{-\zeta}) \bigg) \bigg],  \nonumber 
\end{align}
which is \eqref{asymptotic If}.

Note that for $f \equiv 0$, $(J_{3}(0)+\tilde{J}_{3}(0))/I(0)$ is equal to 
\begin{align*}
& \mathbb{P}\bigg( \hspace{-0.05cm} \Big( |\theta_{j}-\tfrac{\pi}{2}|\leq n^{-\frac{1}{2}+\epsilon} \hspace{-0.065cm} \mbox{ for all } j \hspace{-0.05cm} \in \hspace{-0.05cm} \{1,\ldots,n\} \Big) \hspace{-0.1cm} \mbox{ or } \hspace{-0.05cm} \Big( |\theta_{j}+\tfrac{\pi}{2}|\leq n^{-\frac{1}{2}+\epsilon} \hspace{-0.065cm} \mbox{ for all } j \hspace{-0.05cm} \in \hspace{-0.05cm} \{1,\ldots,n\} \Big) \hspace{-0.05cm} \bigg) \\
& \leq \mathbb{P}\bigg( \hspace{-0.05cm} \Big( |e^{i\theta_{j}}-i|\leq n^{-\frac{1}{2}+\epsilon} \hspace{-0.065cm} \mbox{ for all } j \hspace{-0.05cm} \in \hspace{-0.05cm} \{1,\ldots,n\} \Big) \hspace{-0.1cm} \mbox{ or } \hspace{-0.05cm} \Big( |e^{i\theta_{j}}+i|\leq n^{-\frac{1}{2}+\epsilon} \hspace{-0.065cm} \mbox{ for all } j \hspace{-0.05cm} \in \hspace{-0.05cm} \{1,\ldots,n\} \Big) \hspace{-0.05cm} \bigg).
\end{align*}
For $f \equiv 0$, we also have $J_{3}(0),\tilde{J}_{3}(0)>0$. Thus, by \eqref{prob interpretation}, we get $I(0) = (J_{3}(0)+ \tilde{J}_{3}(0)) \big( 1+ \bigO(e^{-cn^{2\epsilon}}) \big)$, and Theorem \ref{thm:prob} follows.
\subsection{The case of $I(tf)$, $t \in \mathbb{R}$ and $\re f \equiv 0$}
If $f$ is replaced by $tf$ in Subsection \ref{subsection:general f}, and if $\re f \equiv 0$, then the estimates \eqref{bound on J1}, \eqref{sum J2 bound}, \eqref{bound on sum J3} become
\begin{align*}
& |J_{1}| \leq 3 \, (2\pi)^{n} 2^{\beta\frac{n(n-1)}{2}} (\cos \tau)^{\beta n^{1+\epsilon}},  \\
& \sum_{m_{2}=1}^{5 n^{\frac{1+\epsilon}{2}}} |J_{2}(m_{2})| \leq e^{- n \alpha} 2^{\beta\frac{n(n-1)}{2}} \bigg( \frac{8\pi}{\beta n} \bigg)^{\frac{n}{2}} e^{c_{2} \frac{n}{\log n}},  \\
& \sum_{\tilde{m}_{1}=1}^{n}|J_{3}(\tilde{m}_{1})| \leq 2^{\beta\frac{n(n-1)}{2}} \bigg( \frac{8 \pi}{\beta n} \bigg)^{\frac{n}{2}} e^{-c_{3}n^{2\epsilon}}, 
\end{align*}
where $c_{2}>0$ and $c_{3}>0$ are independent of $t$. Furthermore, it directly follows from \eqref{coeff ABCDEF in proof}, \eqref{coeff ABCDEF in proof Re} (with $f$ replaced by $tf$) and Lemma \ref{lemma:ABCD} that the $\bigO(n^{-\zeta})$-terms in \eqref{J30} and \eqref{J30 bis} are uniform for $t$ in compact subsets of $\mathbb{R}$. This proves that the $\bigO(n^{-\zeta})$-terms in \eqref{asymptotic If} are uniform for $t$ in compact subsets of $\mathbb{R}$.

\paragraph{Conflict of interest statement.} The author has no conflict of interest to disclose.

\paragraph{Acknowledgements.} The author is grateful to Brendan McKay and Peter Forrester for useful remarks. Support is acknowledged from the Swedish Research Council, Grant No. 2021-04626.

\end{document}